\documentclass[10pt]{amsart}
\usepackage{indentfirst} 
\usepackage{graphicx} 
\usepackage{hyperref} 
\hypersetup{ 
    colorlinks,
    linkcolor={black},
    citecolor={blue},
    urlcolor={blue}
}
\usepackage{enumitem} 
\usepackage{lastpage} 
\usepackage{biblatex} 
\usepackage{amsmath} 
\usepackage{amsthm} 
\usepackage{amsfonts} 
\usepackage{amssymb}
\usepackage{algorithm}
\usepackage{algpseudocode}
\usepackage{caption}
\usepackage{subcaption}
\usepackage{mathtools}

\usepackage{fancyhdr} 
\usepackage{cleveref}

\newtheorem{thm}{Theorem}

\theoremstyle{definition}

\theoremstyle{remark}

\newcommand{\norm}[1]{\left\Vert #1 \right\Vert}

\newcommand{\C}{\mathbb{C}}
\newcommand{\R}{\mathbb{R}}
\newcommand{\N}{\mathbb{N}}

\newcommand{\Z}{\mathbb{Z}}

\DeclareMathOperator*{\argmax}{arg\,max}

\newcommand{\bs}[1]{\boldsymbol{#1}}
\newcommand{\tbs}[1]{\widetilde{\boldsymbol{#1}}}
\newcommand{\obs}[1]{\overline{\boldsymbol{#1}}}

\begin{document}

\title{Greedy Rational Approximation for Frequency-Domain Model Reduction of Parametric LTI Systems}
\author{Filip B\v{e}l\'{i}k}
\author{Yanlai Chen}
\author{Akil Narayan}

\begin{abstract}
  We investigate model reduction of parametric linear time-invariant (LTI) dynamical systems. When posed in the frequency domain, this problem can be formulated as seeking a low-order rational function approximation of a high-order rational function. We propose to use a standard reduced basis method (RBM) to construct this low-order rational function. Algorithmically, this procedure is an iterative greedy approach, where the greedy objective is evaluated through an error estimator that exploits the linearity of the frequency domain representation. The greedy framework is motivated through theoretical results of rational approximability of functions. This framework provides a principled approach to rational compression of high-order rational functions, and provides a computational pathway for model reduction of parametric LTI systems.
\end{abstract}

\maketitle

\section{Introduction}
Model reduction -- that is, compression techniques that identify low-dimensional representations of models or data -- is becoming an increasingly important computational paradigm in design optimization, uncertainty quantification, and digital twins. We consider the problem of model reduction of parametric linear time-invariant (LTI) dynamical systems by seeking an approximation to the \textit{transfer function} in the frequency domain. By considering the problem in the frequency domain, evolution according to a differential equation is replaced with the problem of function approximation. We approximate the frequency domain transfer function by a rational function so that this corresponds to identifying a lower-dimensional dynamical system in the time domain, hence accomplishing model reduction.

Frequency domain model reduction via rational approximation of the transfer function is a well-explored topic \cite{antoulas_approximation_2005,benner-survey-2015,antoulas_interpolatory_2020-1}. The general motivation we explore, greedy approximation, another well-explored topic, is where a rational function approximant is progressively built by inspecting the maximum error of the current approximant. We expand upon the literature by presenting a computationally feasible procedure for greedy model reduction in the parametric setting. When the LTI system matrices are \textit{affine} in the parameter, our procedure greedily constructs a multivariate rational function approximation of the transfer function. 

\subsection{Main contributions}
We propose and numerically investigate a greedy frequency domain procedure for model reduction of parametric LTI systems. The high-level procedure has ingredients that appear frequently in the literature. We transform the problem of LTI model reduction in the time domain into a low-order rational approximation problem of the frequency-domain transfer function. Under affine assumptions on the parameter-dependence, the frequency domain problem can be written in an algebraic form that is amenable to applying a model reduction approach particularly effective for stationary problems. The Reduced Basis Method (RBM), applied in the context of this problem, greedily constructs a rational function approximation to the transfer function. 

We therefore propose using the RBM methodology to address model reduction for parametric LTI systems. This methodology implicitly constructs a rational approximation to the parametric transfer function. We require some non-standard modifications of traditional RBM pipelines, in particular to account for complex-valued system matrices (which naturally arise in the frequency domain) and to address the non-compact domain of the complex frequency parameter. We present modifications of the successive constraint method (SCM) to address these challenges in addition to a theoretically-motivated selection of parameter points for training.

Our numerical results, which investigate numerous LTI systems, including a time-fractional model, demonstrate that our proposed procedure can be effective, especially so for diffusive models. We also demonstrate that the approach is applicable without serious modification to parametric problems. 

\subsection{Related work}

Our proposed method overlaps with several related methods in the literature that perform frequency domain approximation of the LTI system transfer function. For non-parametric problems, moment matching methods involve forming the Taylor expansion of the transfer function around some fixed frequency value, and choosing trial and test spaces to match the first several Taylor expansion coefficients or moments \cite{benner-survey-2015,feng-posteriori-2017,feng-new-2019,Krylov_MOR}. This is done by forming the test and trial spaces given by Krylov subspaces related to the system arrays \cite{Krylov_MOR}.

One of the most common and popular model reduction techniques in the systems and control theory community is balanced truncation \cite{benner-survey-2015}. If the system is dissipative \cite{qian_model_2022} the LTI system has (time-infinite) reachibility and observatibility Gramians
that can be computationally identified through the numerical solution to dual Lyapunov equations. 
The method of balanced truncation identifies the test and trial spaces through a generalized eigenvalue problem involving these Gramians; these spaces correspond to truncations of a state space transformation of the LTI system that results in equal and diagonal, ``balanced'', Gramians.
More informally,  the balanced state transformation identifies the Hankel singular values of the system, whose decay provides an informative understanding of variable truncation, similar to how traditional singular values inform efficacy of principal component analysis.

Moment matching methods and balanced truncation can be extended to parametric problems by interpolating or aggregating reduced basis across sampled parameter values \cite{benner-survey-2015,Son2021}. However, these methods may require many solves across various parameter values or a strong understanding of which parameter values to sample. 

A data-driven framework leveraging rational interpolation is proposed in \cite{ionita_data-driven_2014}. This framework requires no identification of the full-order model, and instead only evaluations of the transfer function at various parameter and frequency values \cite{ionita_data-driven_2014}. This method allows for the choice of separate reduced orders for each parameter and for the frequency by computing the ranks of appropriate Loewner matrices. Another data-driven parametric model reduction procedure is the \textit{parametric} adaptive Antoulas-Anderson algorithm (p-AAA), which uses transfer function evaluations on a tensor-product grid to construct a multivariate rational approximation to the transfer function \cite{gugercin_paaa_2020}.

Two additional data-driven methods include temporal domain or frequency domain proper orthogonal decomposition (POD) in which solution snapshots are collected and the dominant orthogonal modes are extracted via a singular value decomposition \cite{benner-survey-2015}. In temporal domain POD, the system is simulated across various inputs and parameter values to form snapshots. Frequency domain POD is similar to the greedy method we pose in which frequency domain snapshots are used to form the reduced basis.

Hund et al. posed an optimization-based model order reduction method for parametric LTI systems by performing gradient-based optimization over the space of reduced matrices \cite{hund_optimization-based_2022}. They demonstrate the ability to form reduced transfer functions which are locally $\mathcal{H}_2$-frequency optimal and $\mathcal{L}^2$-parameter optimal.

In contrast to the above approaches, the procedure we propose requires access to the full model matrices but greedily constructs a parametric rational approximation through a residual-based error estimator. 

The remainder of this paper is organized as follows. Section  \ref{sec:bg} establishes the necessary background on parametric LTI dynamical systems and their classical surrogate systems before reformulating the model reduction task as a rational approximation problem in the frequency domain. Section \ref{sec:wgreedy} details the weak greedy framework, discussing the theoretical motivation based on the sectorial properties and describing the essential computational components, including the residual-based error estimator and the requisite modifications to the Successive Constraint Method (SCM). In Section \ref{sec:numerics}, we present numerical experiments on varying benchmarks, including a time-fractional model where the fractional derivative is treated as a parameter, to demonstrate the efficacy and convergence rates of the proposed procedure. Finally, Section \ref{sec:conclusion} offers concluding remarks.

\section{Background}

\label{sec:bg}

\subsection{Linear dynamical systems}
We consider a linear time-invariant (LTI) system with state variable $\bs x(t) \in \C^n$, input/control variable $\bs u(t) \in \C^p$, and output variable $\bs y(t)\in \C^q$, 
\begin{equation}
    \begin{cases}
        \frac{d}{dt} \bs x(t;\bs p) = \bs A(\bs p) \bs x(t;\bs p) + \bs B(\bs p) \bs u(t),\\
        \bs y(t;\bs p) = \bs C(\bs p) \bs x(t;\bs p),\\
        \bs x(0;\bs p) = \bs x_0(\bs p),
    \end{cases}
    \label{eq:lti}
\end{equation}
for $t\geq0$. The variable $\bs p \in \mathcal{P} \subset \C^d$ is a \textit{parameter} that ranges over a closed and bounded set $\mathcal{P}$, and its (time-independent) value affects the dynamics of the system by influencing the matrices $(\bs{A}(\bs{p}), \bs{B}(\bs{p}), \bs{C}(\bs{p}))$. Hence, this is a \textit{parametric} LTI problem.
Several examples of LTI systems are provided at \cite{morwiki}. In the context of model reduction we assume a large state size, $n\gg1$, and small input and output sizes, $p,q\ll n$, where often $p = 1$ and/or $q = 1$ \cite{penzl_algorithms_2006,benner-survey-2015}.

For a fixed parameter value $\bs{p} \in \mathcal{P}$, the system \eqref{eq:lti} is \textit{stable} if the spectrum (set of eigenvalues) of $\bs A(\bs{p})$ has negative real part, and has the stronger property of \textit{dissipativity} if the numerical range of $\bs{A}(\bs{p})$ is a set in $\C$ only containing elements with negative real part.  The numerical range of a matrix $\bs M\in\C^{n\times n}$ is defined as 
\begin{align}\label{eq:numerical-range}
W(\bs M) = \left\{\frac{\bs x^* \bs M \bs x}{\bs x^* \bs x} \; \bigg| \; \bs x \in \C^n \setminus \{\bs 0\} \right\}.
\end{align}
We can generalize these notions to parameter-dependent problems: We say \eqref{eq:lti} is stable if, for all $\bs{p} \in \mathcal{P}$, the spectrum of $\bs A(\bs p)$ lies in the left-half complex plane. It is similarly dissipative if, for all $\bs{p} \in \mathcal{P}$, $W(\bs A(\bs p))$ has only negative real part. We will assume throughout that \eqref{eq:lti} is dissipative in the above sense. Our algorithm does not require dissipativity (only stability), but the motivating theory of our approach does require dissipativity. In particular, applying our algorithm to non-dissipative systems can in principle result in unstable systems. (We have not explored when this happens in practice.) For computational efficacy, we also assume throughout that $\bs{A}$, $\bs{B}$, and $\bs{C}$ have affine dependence on the parameter $\bs{p}$, see \Cref{ssec:affine}.

Problems of the form \eqref{eq:lti} are computationally expensive to solve across different controls, initial conditions, and parameter values. As a result, we seek a reduced system,
\begin{equation}
    \begin{cases}
        \frac{d}{dt}{\tbs x}(t;\bs p) = \tbs A(\bs p) \tbs x(t;\bs p) + \tbs B(\bs p) \bs u(t),\\
        \tbs y(t;\bs p) = \tbs C(\bs p) \tbs x(t;p),\\
        \tbs x(0) = \tbs x_0(\bs p),
    \end{cases}
    \label{eq:lti_red}
\end{equation}
with $\tbs x, \tbs x_0 \in \C^{r}$, $\tbs A \in \C^{r\times r}$, $\tbs B \in \C^{r\times p}$, $\tbs y\in\C^q$, and $\tbs C\in\C^{q\times r}$ with $r\ll n$ such that we can accurately replicate the dynamics of the output variable, $\bs y(t;\bs p) \approx \tbs y(t;\bs p)$. Even if the dynamics of $\bs x$ are not well described in a lower-dimensional space, the dynamics of the smaller dimensional output $\bs{y}$ might be accurately described by such a system.

A common strategy for construction of a reduced order model \eqref{eq:lti_red} from the full LTI system \eqref{eq:lti} is to view these systems in the frequency domain. Upon applying the Laplace transform to \eqref{eq:lti}, assuming sufficient exponential decay of $\bs x$ and homogeneous initial conditions, $\bs x_0 = \bs 0$, and with $\bs{X}$ and $\bs{U}$ the Laplace transformed variables of $\bs x$, $\bs u$, respectively, then \eqref{eq:lti} becomes an algebraic equation with frequency variable, $s \in \C$,
\begin{equation}
    \begin{aligned}
        s \bs X(s;\bs p) &= \bs A(\bs p) \bs X(s;\bs p) + \bs B(\bs p) \bs U(s),\\
        \bs Y(s;\bs p) &= \bs C(\bs p) \bs X(s;\bs p).\\
    \end{aligned}
    \label{eq:lti_freq}
\end{equation}
The $\bs{U}$-to-$\bs{X}$ and $\bs{U}$-to-$\bs{Y}$ mappings are given by $\bs{W}$ and $\bs{H}$, respectively:
\begin{subequations}
\begin{align}
  \bs X(s;\bs p) &= (s\bs I - \bs A(\bs p))^{-1} \bs B(\bs p) \bs U(s) =: \bs W(s;\bs p) \bs U(s),
    \label{eq:W_transfer}\\
    \bs Y(s;\bs p) &= \bs C(\bs p) \bs X(s;\bs p) = \bs C(\bs p) \bs W(s;\bs p) \bs U(s) =: \bs H(s;\bs p) \bs U(s).
    \label{eq:H_transfer}
\end{align}
\end{subequations}
We have defined the parameter dependent transfer functions, $\bs W(s;\bs p) \in \C^{n\times p}$ and $\bs H(s;\bs p) \in \C^{q \times p}$, respectively. We will call $\bs H$ the \textit{output transfer function} and $\bs W$ the \textit{state transfer function}. For fixed $\bs{p}\in\mathcal{P}$, these functions are (generically order-$n$) rational functions of the frequency variable, $s$. Numerous model reduction strategies for fixed-$\bs{p}$ dynamical systems construct order-$r$ $s$-rational function approximations to the output transfer function which amounts to identifying a dimension-$r$ linear dynamical system \eqref{eq:lti_red} that approximates \eqref{eq:lti} \cite{antoulas_approximation_2005,penzl_algorithms_2006,antoulas_interpolatory_2020-1}. If $\bs{A}, \bs{B}, \bs{C}$ are all affine functions of $\bs{p}$, then both $\bs{W}$ and $\bs{H}$ are rational functions of $(s,\bs{p})$.

\subsection{Model reduction and projections}
One way to accomplish the previously described reduction is through a Petrov-Galerkin projection \cite{benner-survey-2015}. If we form trial and test spaces with bases given by the columns of the semiunitary matrices $\bs \Psi \in \C^{n\times r}$ and $\bs \Phi \in \C^{n\times r}$ respectively, then one way to form the reduced system matrices is by $\tbs A = \bs \Psi^* \bs A \bs \Phi$, ${\tbs B} = \bs \Psi^* \bs B$, and ${\tbs C} = \bs C \bs \Phi$ (and $\tbs x_0 = \bs \Phi^* \bs x_0$). The full-order model, \eqref{eq:lti}, is then approximated by $\bs x \approx \bs \Phi \tbs x$ and $\bs y \approx \tbs y$. A Galerkin projection corresponds to $\bs \Psi = \bs \Phi$. 

Stability and dissipativity are preserved under the more restrictive Galerkin projection. General Petrov-Galerkin projections do not enjoy such stability/dissipativity guarantees, but the flexibility of a Petrov-Galerkin approach allows for better reduced approximation.

Error estimates for the reduced model \eqref{eq:lti_red} compared to the original model \eqref{eq:lti} can be derived from function approximation errors in the corresponding transfer functions. The transfer functions of the reduced system are given by,
\begin{align}
    \tbs W(s;\bs p) &= (s\bs I - \tbs A(\bs p))^{-1} \tbs B(\bs p), & \tbs H(s;\bs p) &= \tbs C(\bs p) \tbs W(s;\bs p). 
    \label{eq:reduced-transfer-funcs}
\end{align}
We reiterate that, for fixed $\bs{p}$, the full-order transfer function $\bs{H}$ is a order-$n$ rational function of the frequency parameter $s$, while the reduced order transfer function is an order-$r$ rational function. Thus, one viewpoint of the frequency-domain-based reduced-order modeling is that of rational approximability: how well can we approximate an order-$n$ rational function by an order-$r$ one?

For fixed $\bs p \in \mathcal{P}$, the $\mathcal{H}_\infty$-norm error between the full- and reduced-order transfer functions is given by
\begin{equation}
    \norm{\bs H(\cdot;\bs p) - \tbs H(\cdot;\bs p)}_{\mathcal{H}_\infty} = \sup_{\omega \in \mathbb{R}} \norm{\bs H(i\omega;\bs p) - \tbs H(i\omega;\bs p)}_{2}. \label{h-inf-err}
\end{equation}
One can similarly define a $\mathcal{H}_2$-norm as an $L^2$-type frequency norm. See, e.g., \cite{antoulas_approximation_2005,benner-survey-2015}. The frequency-domain error \eqref{h-inf-err} provides an understanding of time-domain errors. 
\begin{thm}[\cite{benner-survey-2015}]\label{thm:h_infty}
     For a given input $\bs u(t)$ with bounded $L^2(0,\infty)$ norm, and a given parameter vector $\bs p \in \mathcal{P}$, the $L^2(0,\infty)$ error between the outputs of the full-order \eqref{eq:lti} and reduced-order \eqref{eq:lti_red} systems is bounded by
    \begin{equation}
        \norm{\bs y(\cdot;\bs p) - \tbs y(\cdot;\bs p)}_{L^2(0,\infty)} \leq \norm{\bs H(\cdot;\bs p) - \tbs H(\cdot;\bs p)}_{\mathcal{H}_\infty} \norm{\bs u}_{L^2(0,\infty)}. \label{h-inf-l2-comp}
    \end{equation}
\end{thm}
Hence, one way to investigate the accuracy of a reduced order model is by making the $\mathcal{H}_\infty$-error small.
In particular, \Cref{thm:h_infty} implies that the transfer function errors along the imaginary axis in the frequency domain, i.e., $s = i \omega$ with $\omega\in\R$, is sufficient, at least in the context of this discussion. Therefore, in what follows will restrict attention to $s$ along the imaginary axis and write $s = i \omega$.

\section{Weak greedy approximation in the frequency domain}

\label{sec:wgreedy}

\subsection{Stationary Frequency Domain Problem}

For brevity, for $m\in\N$, we let $[m]$ denote the set of positive integers less than or equal to $m$, $[m]=\{1,2,\ldots,m\}$. 
Denoting $\bs w_k(s;\bs p)$ the $k$'th column of the state transfer function $\bs W(s;\bs p)$,
each column, $\bs{w}_k$, is the solution to an $(s,\bs{p})$-parametric linear problem,
\begin{align}
  (s\bs I - \bs A(\bs p)) \bs w_k(s;\bs p) &= \bs b_k(\bs p), & s \in i\R,  \hskip 5pt \bs{p} \in \mathcal{P}, \hskip 5pt k\in [p], \label{eq:transfer_function_stationary}
\end{align} 
where $\bs b_k(\bs p)$ is the $k$'th column of $\bs B(\bs p)$.
As suggested above, one could consider the above problem as a $(s,\bs{p},k)$-parametric problem for vectors $\bs{w}$ indexed by $(s,\bs{p},k)$. To reduce notational clutter moving forward, we codify this now by making the following simplifications:
\begin{subequations}\label{eq:notation-simplification}
\begin{enumerate}
\item We concatenate the original set of parameters $\bs{p} \in \mathcal{P}$, the frequency $s = i \omega$, and the column index $k \in [p]$ into a single parameter vector of size $\dim (\mathcal{P}) + 2 = d+2$. In this section, we will write this concatenated vector as $\bs{P} \in \widetilde{\mathcal{P}}$, i.e.,
  \begin{align}
    \widetilde{\mathcal{P}} \ni \bs{P} \coloneqq (s,\bs{p},k) \in i \R \times \mathcal{P} \times [p].
  \end{align}
\item With the above simplification, we then simply write the columns of the frequency domain solution, $\bs{W}$, as $\bs{w}_k = \bs{w}(\bs{P}) \in \C^n$, which depends on $k$ through $\bs{P}$.  We also write $\bs{b}(\bs{P})$, which encapsulates right-hand side dependence both on the original $\mathcal{P}$-valued parameter and also the column/input index $k$.
\item We write 
  \begin{align}
    \bs M(\bs{P}) = s \bs{I} - \bs{A}(\bs{p})
  \end{align}
  to denote the parameter dependent left-hand side matrix.
\end{enumerate}
\end{subequations}
With these changes, our original problem \eqref{eq:transfer_function_stationary} is equivalent to the following problem:
\begin{equation}
    \bs{M}(\bs P) \bs w(\bs P) = \bs b(\bs P), \label{eq:linearmodel-stationary}
\end{equation}
where the left-hand side matrix, $\bs{M}(\bs P)\in\C^{n\times n}$, is assumed to be invertible over all parameters, $\bs P \in \widetilde{\mathcal{P}} \subset \C^{d+2}$, and we wish to solve for $\bs w(\bs P) = \bs M(\bs P)^{-1} \bs b(\bs P) \in \C^n$. 

Linear model reduction of the system \eqref{eq:linearmodel-stationary} requires a matrix pair $\bs \Phi, \bs{\Psi} \in \C^{n\times r}$. The approximation is the Petrov-Galerkin projection of 
\eqref{eq:transfer_function_stationary}, which yields the system 
\begin{align}
  \tbs M(\bs{P}) \tbs w(\bs P) &= \tbs b(\bs P), \label{eq:transfer_function_stationary_reduced}
\end{align}
where the reduced matrix $\tbs M(\bs{P}) = \bs{\Psi}^\ast \bs{M}(\bs{P}) \bs{\Phi}$, corresponding to the reduced matrices in the reduced LTI problem \eqref{eq:lti_red}, and $\bs w(\bs P) \approx \bs\Phi\tbs w(\bs P)$. 
We note that a $\bs{P}$-uniform ROM error in $\bs{w}$ translates to similar ROM errors for the frequency domain transfer function and also to ROM time domain errors.
\begin{thm}
  Suppose an order-$r$ reduced order model \eqref{eq:transfer_function_stationary_reduced} is formed such that
    \begin{align*}
      \|\bs w(\bs P) - \bs\Phi\tbs w(\bs P)\|_2 &\leq \varepsilon & \forall \; \bs{P} = (s,\bs{p},k) \in i \R \times \mathcal{P} \times [p].
    \end{align*}
    Then, with $\bs{\Phi}\in\C^{n\times r}$ the trial basis for the order-$r$ reduced model, the $\mathcal{H}_\infty$-error between the full-order model \eqref{eq:lti} and the reduced-order model \eqref{eq:lti_red} satisfies
    \begin{align*}
        \left\|\bs H(\cdot;\bs p) - \tbs H(\cdot;\bs p)\right\|_{\mathcal{H}_\infty} &\leq \varepsilon\sqrt{p}\|\bs C(\bs p)\|_2 & \forall \; \bs p\in\mathcal{P},
    \end{align*}
    which, by \Cref{thm:h_infty}, bounds the output error by
    \begin{align*}
        \left\|\bs y(\cdot;\bs p) - \tbs y(\cdot;\bs p)\right\|_{L^2(0,\infty)} &\leq \varepsilon\sqrt{p}\|\bs C(\bs p)\|_2\|\bs u\|_{L^2(0,\infty)} & \forall \; \bs p\in\mathcal{P}.
    \end{align*}
\end{thm}
\begin{proof}
    \begin{align*}
        \left\|\bs H(\cdot;\bs p) - \tbs H(\cdot;\bs p)\right\|_{\mathcal{H}_\infty} &= \sup_{\omega\in\R}\left\|\bs H(i\omega;\bs p) - \tbs H(i\omega;\bs p)\right\|_2\\
        &\leq \sup_{\omega\in\R}\|\bs C(\bs p)\|_2\left\|\bs W(i\omega;\bs p) - \bs \Phi\tbs W(i\omega;\bs p)\right\|_2\\
        &\leq \sup_{\omega\in\R}\|\bs C(\bs p)\|_2\left\|\bs W(i\omega;\bs p) - \bs \Phi\tbs W(i\omega;\bs p)\right\|_F\\
        &= \sup_{\omega\in\R}\|\bs C(\bs p)\|_2\left(\sum_{k=1}^p \left\|\bs w_k(i\omega;\bs p) - \bs\Phi\tbs w_k(i\omega;\bs p)\right\|_2^2\right)^{1/2}\\
        &\leq \varepsilon\sqrt{p}\|\bs C(\bs p)\|_2,
    \end{align*}
    yielding the desired $\mathcal{H}_\infty$ bound. 
\end{proof}

On the stationary parametric problem \eqref{eq:linearmodel-stationary}, we employ the reduced basis method (RBM) \cite{patera_reduced_2007,hesthaven_certified_2016,quarteroni_reduced_2016}, and we describe this standard procedure below. In the ideal setting, one chooses the reduced basis matrices to minimize a worst-case error:
\begin{align}\label{eq:n-width-approx}
  \min_{\bs\Phi, \bs{\Psi} \;:\;\text{dim}(\bs \Phi) = \text{dim}(\bs{\Psi}) = r} \max_{\bs P \in \widetilde{P}} \| \bs{w}(\bs{P}) - \bs{\Phi} \widetilde{\bs{w}}(\bs{P})\|_2,
\end{align}
which is closely related to the Kolmogorov $n$-width. However, since this problem is computationally intractable as posed, one approximates the optimal solution by a greedy iterative search. In particular, the (standard) reduced basis method uses a Galerkin projection ($\bs{\Psi} = \bs{\Phi}$), and ideally forms the reduced basis matrix, $\bs \Phi$, through a (strong) greedy procedure, which iteratively defines the columns of $\bs{\Phi}$ as:
\begin{align}\label{eq:strong-greedy}
  \bs{\Phi}_k &= [ \bs w(\bs P_1), \ldots, \bs w(\bs P_k) ], & \bs{P}_{k+1} \in \argmax_{\bs{P} \in \widetilde{\mathcal{P}}} \| \bs{e}_{\bs \Phi_k}(\bs{P})\|_2,
\end{align}
where $\bs{e}_{\bs \Phi_k}(\bs{P})$ is the error in the order-$k$ ROM:
\begin{align}\label{errore}
  \bs{e}_{\bs \Phi_k}(\bs{P}) &= \bs{w}(\bs{P}) - \bs{\Phi}_k \tbs w(\bs{P}) = \bs w(\bs P) - \bs{\Phi}_k (\bs{\Phi}_k^* \bs M(\bs P) \bs \Phi_k)^{-1} \bs b(\bs P),
\end{align}
and the final, order-$r$ reduced basis is given by $\bs \Phi = \bs \Phi_r$.
In practice, there is an orthogonalization step to mitigate numerical redundancy of the columns of $\bs{\Phi}$.

The innovations that make RBM popular are (i) algorithmic procedures that make the evaluation of the ROM, $\tbs w(\bs{P})$, a computation that depends on $r$ but is entirely independent of the FOM size, $n$, and (ii) a strategy to approximate the norm of $\bs{e}_{\bs\Phi}(\bs{P})$ without actually computing the solution $\bs{w}(\bs{P})$. Before discussing these computational considerations, we provide a short background on theory motivating greedy-type algorithms of the above form.

\subsection{Theoretical Motivation}\label{ssec:theory}
We briefly discuss some theoretical results that motivate using RBM for rational approximation of transfer functions. 

One initial set of results establishes theoretical feasibility. The best possible approximation by a rank-$r$ linear reduced model is described by the Kolmogorov $n$-width; this $n$-width is a modification of \eqref{eq:n-width-approx}, where the argument to the norm is replaced by the best-$\ell^2$ approximation of $\bs{w}(\bs{P})$ in $\mathrm{range}(\bs{\Phi})$. Fast decay of the $n$-width suggests that linear model reduction can be effective if a near-optimal subspace can be found. It is known that the $n$-width decays quickly in $n$ (exponentially, root-exponentially, or algebraically) for $\bs{P}$-smooth and -affine operators \cite{cohen_approximation_2015,cohen_kolmogorov_2016}, which is the case for our frequency-domain problem \eqref{eq:linearmodel-stationary}. Building on this, strong and weak greedy schemes to identify a linear subspace for model reduction achieve error decay comparable to the $n$-width decay \cite{maday_global_2002,binev_convergence_2011,buffa_priori_2012,devore_greedy_2013}. Under our assumption of dissipativity of the system \eqref{eq:lti}, these established estimates for reduced basis approximation hold for our problem \eqref{eq:linearmodel-stationary}.

A somewhat more technical and precise motivation for our procedure is provided by a result in the upcoming preprint \cite{chen_hinfty_2026}. To articulate this result, we require some additional notation. Consider the single-input, single-output nonparametic version of problem \eqref{eq:lti}, so that \eqref{eq:linearmodel-stationary} has only one parameter $\bs{P} = \omega \in \R$. Suppose that $\bs{A}$ is sectorial, i.e., that the numerical range $W(\bs{A})$ defined in \eqref{eq:numerical-range} is contained in a left half-plane sector in $\C$. 
Preciously, a matrix $\bs A$ is sectorial if there exist constants $m\geq0$ and $\omega\in\R$ such that
\begin{align}\label{eq:sectorial}
W(\bs A) \subset \Sigma_m(i\omega)\setminus\{i\omega\},
\end{align}
where $\Sigma_m(i\omega)$ is the shifted cone in the complex plane visualized in \Cref{fig:sectorial_region} defined as
\[\Sigma_m(i\omega) = \left\{i\omega + \alpha(-1 + i m) + \beta(-1 - i m) \in \C \;\mid\; \alpha,\beta \geq 0 \right\}.\]
We assume $m$ to be the smallest possible value such that \eqref{eq:sectorial} holds.
\begin{figure}[htbp]
    \centering
    \includegraphics[width=0.6\textwidth]{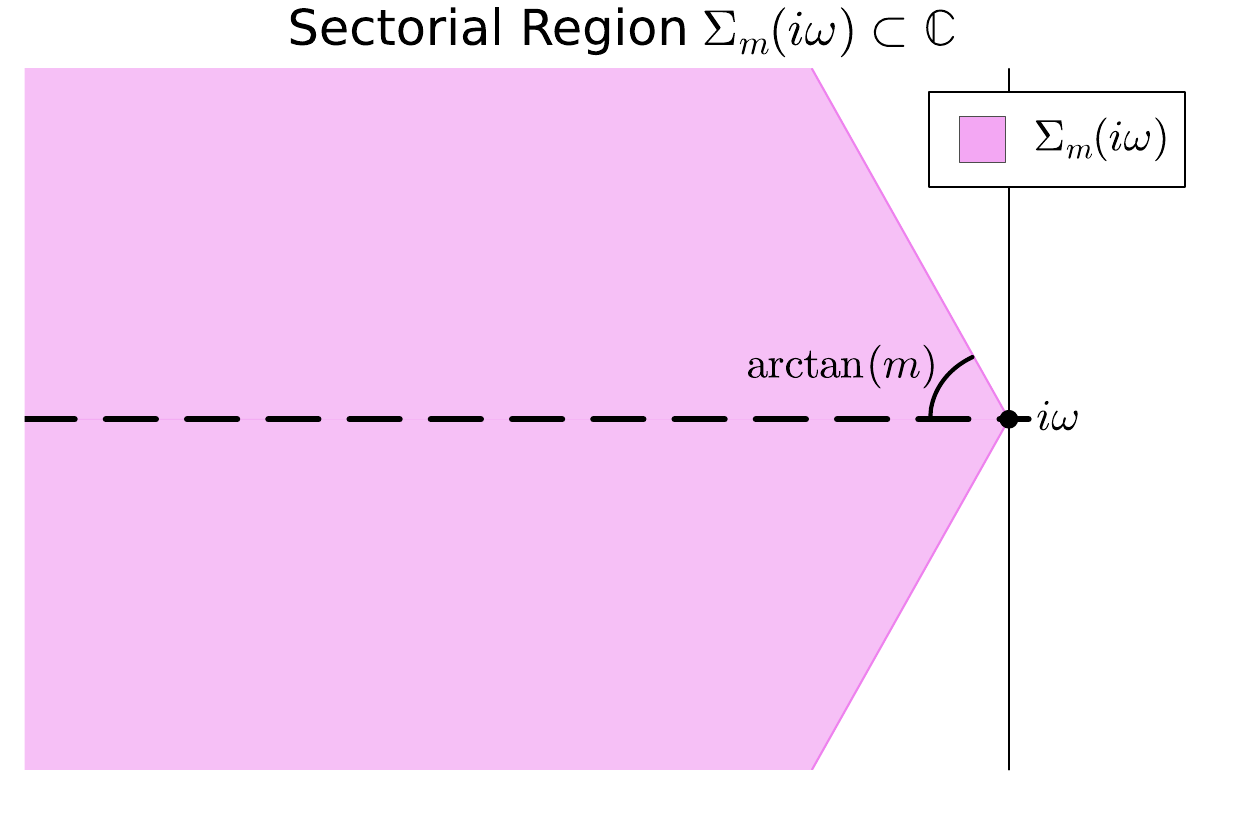}
    \caption{Region in definition of sectorial matrix.}
    \label{fig:sectorial_region}
\end{figure}
The following upcoming result in \cite{chen_hinfty_2026} then follows: if an RBM approximation is constructed in the strong greedy sense \eqref{eq:strong-greedy}, then 
\begin{align}\label{eq:sectorial-estimate}
\inf_{\bs \Phi \in \C^{n\times r}} \|\bs H(\cdot) - \tbs H(\cdot)\|_{\mathcal{H}_\infty} \lesssim \exp\left(\frac{-r}{m}\right).
\end{align}
In particular, this estimate identifies that the value of $m$ quantitatively indicates the efficacy of linear model reduction. Small values of $m$ suggest the poles of the system do not cluster near the imaginary axis, resulting in effective linear model reduction. Large values of $m$ suggest that the poles cluster along (or at least the numerical range nearly abuts) the imaginary axis, making linear model reduction difficult. This intuitive interpretation does translate into practical behavior, see \Cref{fig:sectorial} for a simple example of Hankel singular-value decay for simple, artificial, dimension-1000 LTI systems with varying sectorial parameter, $m$.
\begin{figure}[htbp]
    \centering
    \includegraphics[trim={0 0 4.8cm 0},clip,height=0.35\textwidth]{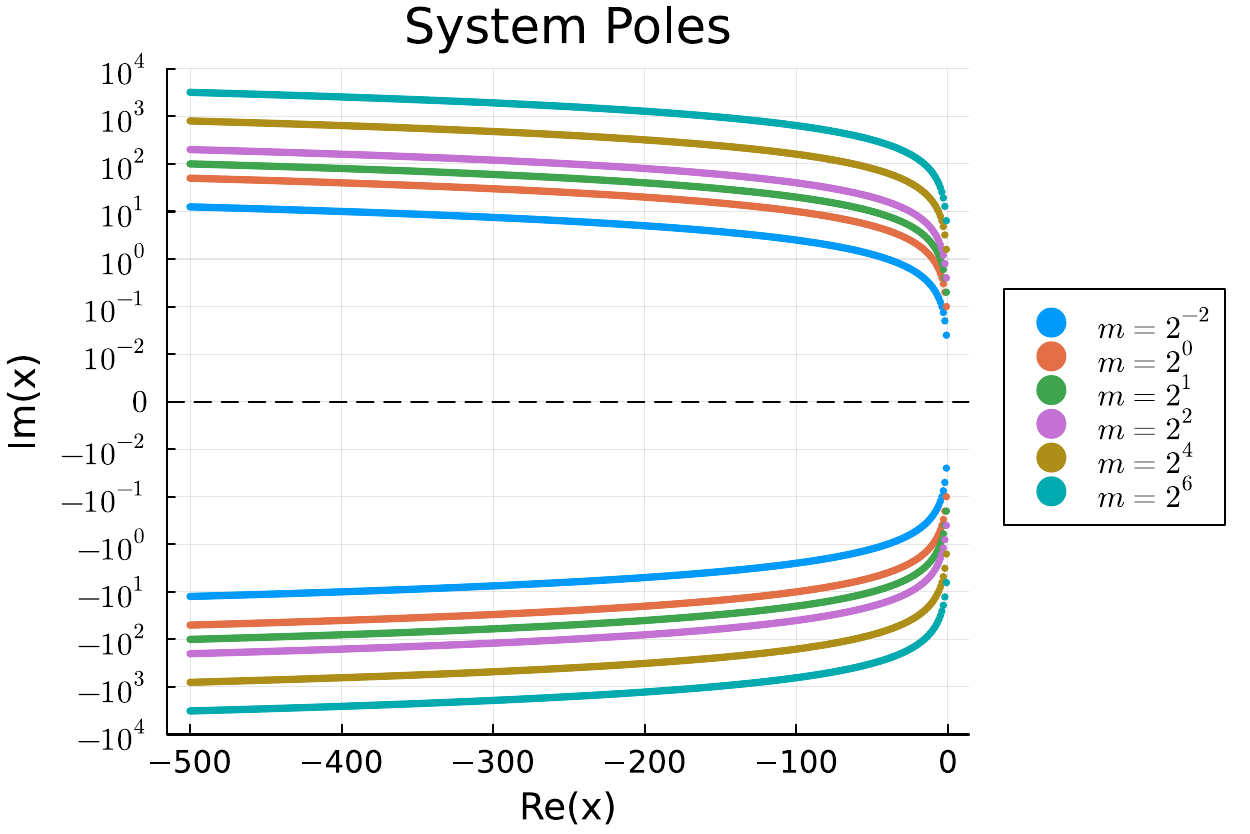}
    \includegraphics[height=0.35\textwidth]{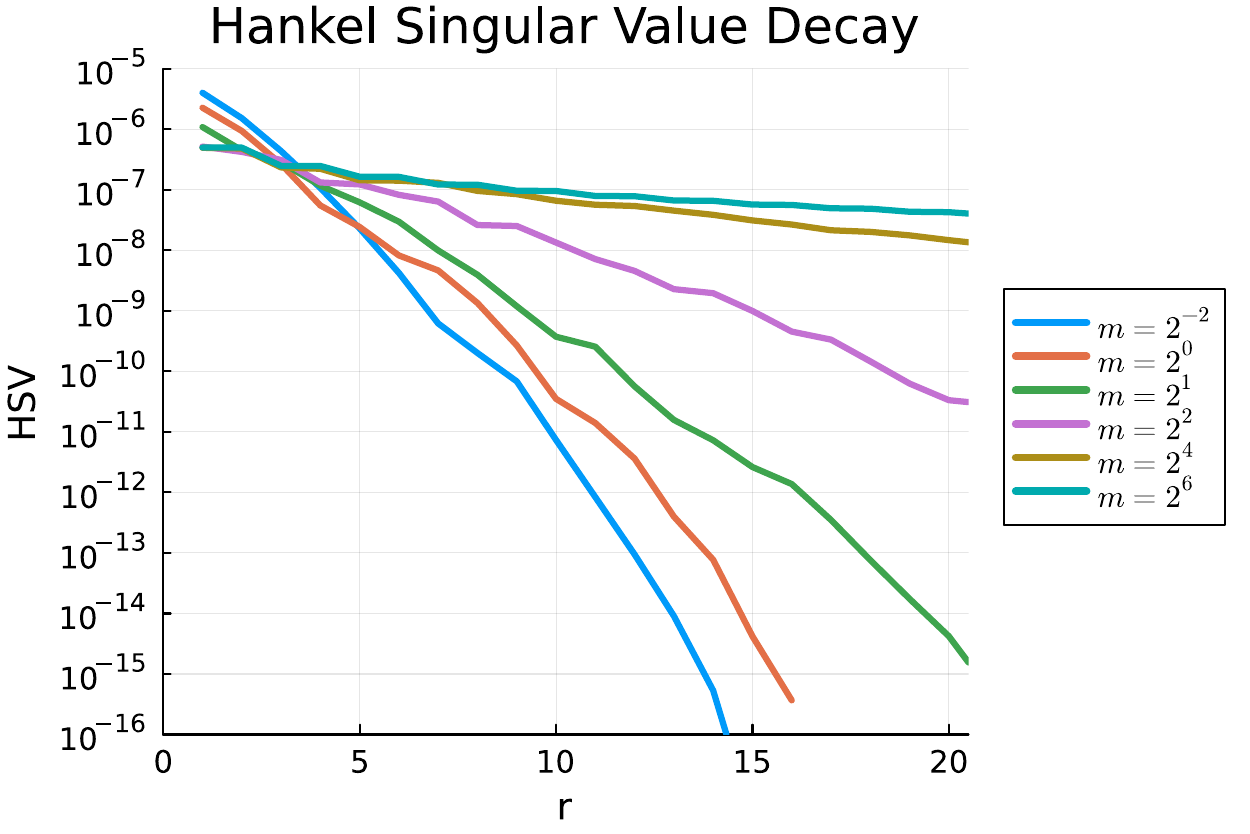}
    \caption{Slowly decaying Hankel singular values for problems with larger sectorial parameter, $m$.}
    \label{fig:sectorial}
\end{figure}
The proof of the result \eqref{eq:sectorial-estimate} is accomplished by deriving bounds for the $n$-width, and involves a construction of a partition of the imaginary axis using intervals with geometrically increasing length; this informs a numerical discretization of the frequency parameter that we exploit later, cf. \Cref{sssec:results-nnscm}.
The theory above provides strong motivation for the approach \eqref{eq:strong-greedy}. We describe next the ingredients of RBM that make the computations practical.

\subsection{Efficient evaluation of the ROM}\label{ssec:affine}
To evaluate the ROM efficiently, we make the standard RBM assumption that the LTI system arrays have affine parameter dependence,
\begin{align}
    \bs A(\bs p) &= \sum_{j=1}^{Q_A} \theta^A_j(\bs p) \bs A_j, & \bs b(\bs p) &= \sum_{j=1}^{Q_b} \theta^b_j(\bs p) \bs b_j,
    \label{affine-dep-lti}
\end{align}
  Note that under this affine assumption, the reduced order solution $\tbs w$ in \eqref{eq:transfer_function_stationary_reduced} (or the full-order model approximation $\bs{\Phi} \tbs w$) are \textit{multivariate rational} functions of $(\bs{p},\omega)$. Hence, the idealized RBM procedure \eqref{eq:strong-greedy} greedily constructs a multivariate rational approximation to the transfer function.

We additionally assume the affine coefficients are real-valued, i.e., $\theta_j^A(\bs p), \theta_j^b(\bs p)\in \R$ for all $\bs p \in \mathcal{P}$. This real-valued assumption is important for computational evaluation of an error estimate in the RBM method, as described later. General complex-valued affine dependence  can in general be expressed with real-valued coefficients by splitting up the real and imaginary parts of each coefficient. For problems that do not have such a decomposition, the empirical interpolation method is one common approach to approximate a non-affine parametric operator with an affine one \cite{ohlberger-reduced-2016}. 

Because we focus on a purely imaginary frequency variable, we write $s=i\omega$ with $\omega \in \R$, so the matrix \eqref{eq:transfer_function_stationary} also has an affine parameter decomposition with real coefficients,
\[\bs M(\bs P) = s\bs I - \bs A(\bs p) = \omega (i \bs I) - \sum_{j=1}^{Q_A} \theta^A_j(\bs p) \bs A_j.\]
Linearity implies that the reduced system arrays, $\tbs M(\bs P)$ and $\tbs b(\bs{P})$, also have affine dependence,
\begin{equation}
    \begin{alignedat}{4}
      \tbs M(\bs P)  &= \sum_{j=1}^{Q_M} \theta_j^M(\bs{P}) \tbs M_j,\quad &\tbs M_j &= \bs \Phi^* \bs M_j \bs \Phi,\quad & j &=1,2,\ldots,Q_M,\\
        \tbs b(\bs P)  &= \sum_{j=1}^{Q_b} \theta_j^b(\bs{P}) \tbs b_j,\quad &\tbs b_j &= \bs \Phi^* \bs b_j,  & j&=1,2,\ldots,Q_b,
    \end{alignedat} 
\end{equation}
which makes ``online'' computation of $\tbs w$ efficient (requiring only dense linear algebra operations on size-$r$ vectors and matrices) for any $\bs{p}$ since the affine matrices and vectors $\tbs M_i$ and $\tbs b_i$, respectively, are parameter-independent and can be precomputed.

Hence, after we have constructed $\bs \Phi$ by a linear weak-greedy reduced basis method for \eqref{eq:transfer_function_stationary}, we use $\bs \Phi$ to Galerkin project the original system, \eqref{eq:lti}, and then forward-iterate it in time. This weak greedy method can account for affine parameter dependence of the original LTI system in a relatively straightforward manner.

\subsection{Evaluation of the error}

What remains is to describe how to perform the greedy maximization \eqref{eq:strong-greedy} without knowledge of the solution $\bs{w}$ at all parameter values.
For the model \eqref{eq:linearmodel-stationary}, define the residual as
\begin{equation}
    \bs r_{\bs \Phi}(\bs P) = \bs M(\bs P) \bs e_{\bs \Phi}(\bs P) =  \bs b(\bs P) - \bs M(\bs P)\bs \Phi \tbs w(\bs P), \label{errorr}
\end{equation}
which no longer depends on the full-order solution, $\bs w(\bs P)$, and instead only the reduced order solution, which can be evaluated with $n$-independent complexity.  Inserting the residual in \Cref{errorr} into \Cref{errore} and using the Cauchy-Schwarz inequality, we have,
\begin{equation}
    \Delta_{\bs \Phi}(\bs P) \coloneqq \frac{\|\bs r_{\bs \Phi}(\bs P)\|_2}{\sigma_{\text{min}}(\bs P)} = \|\bs M(\bs P)^{-1}\|_2 \|\bs r_{\bs \Phi}(\bs P)\|_2 \geq \|\bs e_{\bs \Phi}(\bs P)\|_2, \label{upperbound-estimator}
\end{equation}
where $\sigma_{\text{min}}(\bs P)$ is the minimum singular value of $\bs M(\bs P)$ \cite{quarteroni_reduced_2016}. The quantity $\Delta_{\bs \Phi}(\bs{P})$ is an upper-bound error surrogate. Note that $\sigma_{\text{min}}(\bs P) > 0$ due to our assumed invertibility of $\bs M(\bs P)$. Under additional regularity/boundedness assumptions on $\bs{M}$, the error surrogate, $\Delta_{\bs \Phi}(\bs P)$, satisfies a norm-equivalence with the norm of the error. Define the parameter independent terms,
\begin{align*}
    \sigma_{\text{inf}} &= \inf_{\bs P \in \widetilde{\mathcal{P}}} \sigma_{\text{min}}(\bs P), &  \sigma_{\text{sup}} &= \sup_{\bs P \in \widetilde{\mathcal{P}}} \sigma_{\text{max}}(\bs P),
\end{align*}
as $\bs{P}$-uniform bounds on the extremal singular values of $\bs{M}$. Our additional uniform regularity assumption is that $0 < \sigma_\text{min}$ and $\sigma_\text{max} < \infty$, in which case a lower-bound of the norm-equivalence statement takes the form
\begin{equation}
    \frac{\sigma_{\text{inf}}}{\sigma_{\text{sup}}}\Delta_{\bs \Phi}(\bs P) \leq \frac{\|\bs r_{\bs \Phi}(\bs P)\|_2}{\sigma_{\text{max}}(\bs P)}  = \frac{\|\bs r_{\bs \Phi}(\bs P)\|_2}{\|\bs M(\bs P)\|_2}\leq \|\bs e_{\bs \Phi}(\bs P)\|_2.
    \label{lowerbound-estimator}
\end{equation}
Note that $\sigma_{\min}(\bs{P})$ and $\sigma_{\max}(\bs{P})$ are discrete analogues of the inf-sup and continuity constants for operators, respectively.

Equations \eqref{upperbound-estimator} and \eqref{lowerbound-estimator} illustrate that the error estimator $\Delta_{\bs \Phi}(\bs P)$ is $\bs{P}$-equivalent to the norm of the error, $\|\bs e_{\bs \Phi}(\bs P)\|_2$, under the uniform regularity assumption. RBM uses the greedy procedure that replaces $\|\bs{e}_{\bs \Phi}(\bs{P})\|_2$ in \eqref{eq:strong-greedy} with an upper-bound estimator $\Delta_{\bs{\Phi}}(\bs{P})$: 
\begin{align}\label{eq:rbm}
  \bs{\Phi}_k &= [ \bs w(\bs P_1), \ldots, \bs w(\bs P_k) ], & \bs{P}_{k+1} \in \argmax_{\bs{P} \in \widetilde{\mathcal{P}}} \Delta_{\bs{\Phi}}(\bs{P}).
\end{align}
This is a \textit{weak} greedy procedure under the inequalities \eqref{upperbound-estimator} and \eqref{lowerbound-estimator}. In the context of model reduction, weak greedy procedures are known to produce ROM's of comparable accuracy to strong greedy methods \cite{binev_convergence_2011,devore_greedy_2013}.

RBM (again, ideally) uses the estimator $\Delta_{\bs\Phi}(\bs{P})$ defined in \eqref{upperbound-estimator} to perform the optimization \eqref{eq:rbm}. The numerator is the residual norm, which can be stably computed using only the (inexpensive) ROM solution \cite{buhr_numerically_2014,chen-robust-2019}. However, direct computation of $\sigma_{\min}(\bs{P})$ is avoided over all $\bs{P} \in \widetilde{\mathcal{P}}$, and this quantity is instead replaced by yet another computational estimate, in particular, by a \textit{lower} bound so that the resulting error estimate remains an upper bound.

In summary, the computation of the error estimator, $\Delta_{\bs \Phi}(\bs P)$, does not require solving the full-order system, and since the error surrogate acts as an upper-bound to the error \eqref{upperbound-estimator}, the error estimator provides an inexpensive error certificate. Additionally, under the uniform regularity assumption, since the error estimate is equivalent (in the sense of norms) to the norm of the error, $\Delta_{\bs \Phi}(\bs P)$ is a useful \textit{a posteriori} ROM error estimate \cite{quarteroni_reduced_2016}. 

All the above describes the core RBM algorithm. However, to computationally estimate $\sigma_{\min}(\bs{P})$, we describe one more tool: the Successive Constraint Method.

\subsection{Estimating $\sigma_{\min}(\bs{P})$: The Successive Constraint Method}

In order to compute a lower bound for the stability factor $\sigma_\text{min}(\bs P)$, we employ the Successive Constraint Method (SCM) \cite{huynh-successive-2007}. We describe this procedure below, including some augmentations required for our core problem.

\subsubsection{Standard SCM}\label{sssec:standard-scm}
The initial SCM algorithm from \cite{huynh-successive-2007} produces lower-bound estimates of the stability factor through iteratively and greedily constrained solutions to linear programs whose evaluation entails computational complexity that is independent of the full-order dimension $n$. (Although, the ``offline'' greedy initialization of the method has $n$-dependent complexity.)  These methods typically involve reformulating the stability factor problem in its variational form,
\begin{equation}
    \label{eq:stability-variational}
    \sigma_\text{min}(\bs P)^2 = \inf_{\bs v \in \C^n\setminus \{\bs 0\}} \frac{\bs v^* \bs M(\bs P)^* \bs M(\bs P) \bs v}{\bs v^*\bs v}.
\end{equation}
SCM computes lower and upper bounds to this quantity by exploiting the parameter-affine structure of $\bs M$. Let $(\cdot)^\text{H}$ denote the Hermitian part of a matrix, and define parameter-independent Rayleigh quotients of Hermitian parts of quadratic products of the affine matrices of $\bs M$,
\begin{align*}
  R_{j,m}(\bs{v}) &\coloneqq \frac{\bs v^* (\bs M_j^* \bs M_m)^\text{H} \bs v}{\bs v^*\bs v}.
\end{align*}
Then by expanding \eqref{eq:stability-variational} using the parameter-affine decomposition of $\bs M$, we have,
\begin{equation}
    \label{eq:stability-variational-affine}
    \begin{aligned}
      \sigma_\text{min}(\bs P)^2 &= \inf_{\bs v \in \C^n\setminus \{\bs 0\}} \sum_{j=1}^{Q_M} \sum_{m=j}^{Q_M} (2 - \delta_{j,m}) \theta_j^M(\bs P) \theta_m^M(\bs P) R_{j,m}(\bs{v}) \\ 
        &=: \inf_{\bs v \in \C^n \setminus \{\bs 0\}}\mathcal{J}\left(\bs P, \left(R_{j,m}(\bs{v})\right)_{1\leq j\leq m\leq Q_M}\right).
    \end{aligned}
\end{equation}
For notational brevity, we write,
\begin{align}
  \bs y(\bs v) &= [R_{1,1}(\bs{v}),\ldots,R_{j,m}(\bs v)], & \sigma_{\min}(\bs{P})^2 &= \inf_{\bs{v} \in \C^n\backslash\{\bs{0}\}} \mathcal{J}\left(\bs{P}, \bs{y}(\bs{v}) \right).
\end{align}
SCM proceeds iteratively as follows: suppose that a size-$k$ set $\mathcal{C}_k \subset \widetilde{\mathcal{P}}$ has been determined, and that for every parameter in this set the (exact) stability factor and its corresponding extremal singular vector have been identified:
\begin{align*}
  \mathcal{C}_k &= \left\{\bs{P}_j \right\}_{j \in [k]}, & 
  \sigma_{\min}(\bs{P}_j)^2 &= \frac{\|\bs{M}(\bs{P}_j) \bs{v}(\bs{P}_j) \|^2_2}{\|\bs{v}(\bs{P}_j)\|_2^2}.
\end{align*}
Then an upper bound for the stability factor can be identified by restricting the feasible set of values of $\bs{v}$,
\begin{equation}
    \label{eq:stability-ub}
    \sigma_\text{min}^\text{UB}(\bs P, \mathcal{C}_k)^2 \coloneqq \min_{\bs P_\ell \in \mathcal{C}_k} \mathcal{J}\left(\bs P, \bs y(\bs{v}(\bs P_\ell))\right) \geq \sigma_\text{min}(\bs P)^2.
\end{equation}
To obtain a lower bound, we note that, for any nontrivial $\bs{v}$, we have $R_{j,m}(\bs{v}) \in \left[\lambda_{\mathrm{min}}(\bs{M}_j^\ast \bs{M}_m)^\text{H}, \lambda_{\mathrm{max}}(\bs{M}_j^\ast \bs{M}_m)^\text{H}\right]$. Therefore, a lower bound for $\sigma_{\min}$ can be formed by relaxing the feasible set for $\bs{y}$ and allowing it to range over all possible allowable values:
\begin{equation}
    \label{eq:stability-lb}
    \begin{aligned}
        \widetilde{\sigma}_\text{min}^\text{LB}(\bs P, \mathcal{C}_k)^2 = \min \bigg\{& \mathcal{J}(\bs P, \bs y) \;\bigg|\\
        &\lambda_\text{min}((\bs M_j^* \bs M_m)^\text{H}) \leq y_{j,m} \leq \lambda_\text{max}(\bs M_j^* \bs M_m)^\text{H},\\
        &\mathcal{J}(\bs P_\ell, \bs y) \geq \sigma_\text{min}(\bs P_\ell)^2 \quad \forall \; \bs P_\ell \in \mathcal{N}(\bs P) \subset \mathcal{C}_k\bigg\}.
    \end{aligned}
\end{equation}
Note the slight abuse of notation that $\widetilde{\sigma}_\text{min}^\text{LB}(\bs P, \mathcal{C}_k)^2$ represents a lower-bound approximant to a squared minimum singular value, however, the quantity need not be positive before performing sufficient constraining. We will still write this quantity as a square to keep consistent with the fact that it is a lower-bound approximant to a positive, squared quantity.
Here, $\mathcal{N}(\bs P)$ is a neighborhood of $\bs{P}$, cf. \Cref{ssec:parameter-discretization}. In practice the parameter set $\widetilde{\mathcal{P}}$ is discretized so that in this case $\mathcal{N}(\bs P)$ is typically chosen to be the $\text{min}(k,M_\alpha)$ Euclidean nearest neighbors to $\bs P$ in $\mathcal{C}_k$, where $M_\alpha$ is a hyperparameter to set. Note that, if the extremal eigenvalues of $(\bs{M}_j^\ast \bs{M}_m)^H$ are precomputed, then the lower bound minimization problem above is independent of the FOM dimension $n$. From this quantity, the lower-bound approximant to the minimal singular value is given by
\begin{align*}
\sigma_\text{min}^\text{LB}(\bs P, \mathcal{C}_k) = \sqrt{\max(0,\widetilde{\sigma}_\text{min}^\text{LB}(\bs P, \mathcal{C}_k)^2)}.
\end{align*}

With all of these pieces together, the SCM relative error is defined as the spread of the normalized upper- and lower-bound intervals,
\begin{equation}
    \label{eq:scm-relative-gap}
    \epsilon(\bs P, \mathcal{C}_k) = \frac{\sigma_\text{min}^\text{UB}(\bs P, \mathcal{C}_k)^2 - \widetilde{\sigma}_\text{min}^\text{LB}(\bs P, \mathcal{C}_k)^2}{\sigma_\text{min}^\text{UB}(\bs P, \mathcal{C}_k)^2} \geq 0,
\end{equation}
and SCM proceeds inductively,
\begin{align}\label{eq:scm-induction}
  \epsilon_{k} &= \max_{\bs P \in \widetilde{\mathcal{P}}} \epsilon(\bs P, \mathcal{C}_k), & \bs P_{k+1} &= \argmax_{\bs P \in \widetilde{\mathcal{P}}} \epsilon(\bs P, \mathcal{C}_k), & \mathcal{C}_{k+1} &= \mathcal{C}_k \cup \{\bs P_{k+1}\},
\end{align}
and terminates when $\epsilon_k$ is below a desired threshold.

\subsubsection{Natural-norm SCM}\label{sssec:nnscm}
The standard SCM method may struggle when the cross-terms, $(\bs M_j^* \bs M_m)^\text{H}$ for $j\neq m$, have both positive and negative spectrum both of large magnitude. (In this case the relaxed $\bs{y}$ variables in \eqref{eq:stability-lb} make $\mathcal{J}$ poorly represent the Rayleigh quotient of the original operator.)
This situation arises fairly commonly in practice (e.g., in many of the examples we consider in this manuscript). As a result, $\epsilon$ remains around or above the value of 1 for much of the procedure.

The natural norm successive constraint method (NNSCM) in \cite{chen_certified_2016} improves the traditional SCM approach.
Recall the variational form \eqref{eq:stability-variational}; in NNSCM, this quantity is ``linearized'' about a fixed parameter, $\obs P$, and normalized with respect to the operator natural norm at this parameter, introducing the new quantity
\begin{equation}
    \label{eq:nnscm-beta}
    \beta(\bs P, \obs P) = \inf_{\bs v \in \C \setminus \{\bs 0\}} \frac{\bs v^*(\bs M(\obs P)^* \bs M(\bs P))^\text{H}\bs v}{\bs v^* \bs M(\obs P)^*\bs M(\obs P) \bs v},
\end{equation}
which implies that $\beta(\obs P,\obs P) = 1$, and if $\bs P \mapsto \bs{M}(\bs{P})$ is continuous, then $\beta(\bs P,\obs P) \geq 0$ for $\bs P$ ``near'' $\obs P$. The solution to the optimization problem \eqref{eq:nnscm-beta} is given by a generalized eigenvalue problem, so that there is a nontrivial infimum-achieving vector $\bs{v}(\bs P,\obs P)$ satisfying,
\begin{align}\label{eq:beta-geneig-problem}
  \beta(\bs P,\obs P) = \frac{\bs v(\bs P,\obs P)^*(\bs M(\obs P)^* \bs M(\bs P))^\text{H}\bs v(\bs P,\obs P)}{\left\|\bs{M}(\obs P) \bs{v}(\bs P,\obs P)\right\|_2^2}.
\end{align}
The main utility of $\beta$ is that it satisfies the inequality
\begin{equation}
    \beta(\bs P,\obs P) \sigma_\text{min}(\obs P) \leq \sigma_\text{min}(\bs P),
\end{equation}
i.e., $\beta$ provides a quantitative way to generate a lower bound for $\sigma_{\text{min}}(\bs{P})$. Hence, in NNSCM we compute lower-bound predictions to $\beta$ via linear programs and use this to compute a lower bound for $\sigma_\text{min}(\bs P)$.

To construct a linear program for $\beta$ in \eqref{eq:nnscm-beta}, we exploit the $\bs{P}$-affine dependence of $\bs{M}$, similar to standard SCM:
\begin{equation}
    \label{eq:nnscm-beta-affine}
    \begin{aligned}
        \beta(\bs P, \obs P) &= \inf_{\bs v \in \C \setminus \{\bs 0\}} \sum_{j=1}^{Q_M} \theta_j^M(\bs P) \frac{\bs v^*(\bs M(\obs P)^* \bs M_j)^\text{H}\bs v}{\bs v^* \bs M(\obs P)\bs M(\obs P) \bs v}\\
        &=: \inf_{\bs v \in \C \setminus \{\bs 0\}}\mathcal{J}_\text{NN}\left(\bs P, \bs{z}(\bs{v}, \obs P) \right), 
    \end{aligned}
\end{equation}
where 
\begin{equation}
    \bs z(\bs{v},\obs P) =  \left(\frac{\bs v^*(\bs M(\obs P)^* \bs M_j)^\text{H}\bs v}{\bs v^* \bs M(\obs P)\bs M(\obs P) \bs v}\right)_{1\leq j\leq Q_M}.
\end{equation}
Next, we introduce a $\obs P$-dependent finite parameter set $\mathcal{C}_k(\obs P)$ to approximate upper and lower-bound predictions to $\beta$. Like $\mathcal{C}_k$ for standard SCM, the set $\mathcal{C}_k(\obs P)$ corresponds to a greedily constructed parameter set. At each $\bs p_\ell \in \mathcal{C}_k(\obs p)$, we identify a corresponding $\bs v(\bs P_\ell, \obs P)$, defined through the generalized eigenvalue solution \eqref{eq:beta-geneig-problem}.
An upper bound for $\beta$ is given similarly to standard SCM by restricting the feasible set,
\begin{equation}
    \label{eq:nnscm-beta-ub}
        \beta^\text{UB}(\bs P, \obs P, \mathcal{C}_k(\obs p)) = \min_{\bs P_\ell \in \mathcal{C}_k(\obs P)} \mathcal{J}_\text{NN}(\bs P, \bs z(\bs{v}(\bs P_\ell, \obs P),\obs P)).
\end{equation}
Similarly, a lower bound $\beta(\bs P, \obs P)$ is identified through relaxing the feasible values of $\bs{z}$,
\begin{equation}
    \label{eq:nnscm-beta-lb}
    \begin{aligned}
        \beta^\text{LB}(\bs P, \obs P, \mathcal{C}_k(\obs P)) = \min \bigg\{& \mathcal{J}_\text{NN}(\bs P, \bs z) \;\bigg|\\
        &\frac{-\sigma_\text{max}(\bs M_j)}{\sigma_\text{min}(\bs M(\obs P))} \leq z_{j} \leq \frac{\sigma_\text{max}(\bs M_j)}{\sigma_\text{min}(\bs M(\obs P))},\\
        &\mathcal{J}_\text{NN}(\bs P_\ell, \bs z) \geq \beta(\bs P_\ell,\obs P) \; \forall \; \bs P_\ell \in \mathcal{N}(\bs P) \subset \mathcal{C}_k(\obs P)\bigg\},
    \end{aligned}
\end{equation}
where the relaxed bounding constraints on $z_j$ follow from 
\begin{align*}
    \left|\frac{\bs v^*(\bs M(\obs P)^* \bs M_j)^\text{H}\bs v}{\bs v^* \bs M(\obs P)\bs M(\obs P) \bs v}\right| &= \frac{\left|\bs v^* (\bs M(\obs P)^* \bs M_j)^\text{H}\bs v\right|}{\|\bs M(\obs P)\bs v\| \|\bs v\|} \frac{\|\bs v\|}{\|\bs M(\obs P)\bs v\|}\\
    &\leq \frac{\|\bs M_j \bs v\|}{\|\bs v\|}\frac{\|\bs v\|}{\|\bs M(\obs P)\bs v\|}\\
    &\leq \frac{\sigma_\text{max}(\bs M_j)}{\sigma_\text{min}(\bs M(\obs P))}.
\end{align*}
Again, $\mathcal{N}(\bs P)$ is chosen to be the $\text{min}(k,M_\alpha)$ Euclidean nearest neighbors to $\bs P$ in $\mathcal{C}_k(\obs P)$.

Similar to standard SCM, fixing $\obs P$ and $\mathcal{C}_k(\obs P)$, one computes the NNSCM relative gap as
\begin{equation}
    \epsilon_\beta(\bs P, \obs P, \mathcal{C}(\obs P)) = \frac{\beta^\text{UB}(\bs P, \obs P, \mathcal{C}_k(\obs P)) - \beta^\text{LB}(\bs P, \obs P, \mathcal{C}_k(\obs P))}{\beta^\text{UB}(\bs P, \obs P, \mathcal{C}_k(\obs P))}.
\end{equation}
Note that the upper-bound prediction need not be positive, so the relative gap may be negative for $\bs P \neq \obs P$. However, importantly, for $\bs P$ ``near'' $\obs P$, this quantity can be positive and is then used for further enriching $\mathcal{C}_k(\obs P)$.

Once one has repeated the above process for a set $\mathcal{C}_k = \{\obs P_\ell\}_{\ell=1}^k$, one can compute an upper-bound prediction to $\sigma_\text{min}(\bs M(\bs P))$ by \eqref{eq:stability-ub} and a lower-bound prediction as
\begin{equation}
    \label{eq:nnscm-sigma-lb}
    \sigma^\text{LB}_\text{min}(\bs P,\mathcal{C}_k) = \max_{\obs P_\ell \in \mathcal{C}_k} \beta^\text{LB}(\bs P, \obs P_\ell, \mathcal{C}_k(\obs P_\ell)) \sigma_\text{min}(\obs P_\ell),
\end{equation}
from which one can compute a relative gap as in \eqref{eq:scm-relative-gap}. The full NNSCM algorithm is given in \Cref{alg:nnscm}, \Cref{sec:nnscm-alg}.

\subsection{Parameter discretization}\label{ssec:parameter-discretization}
The optimizations in \Cref{eq:rbm,eq:stability-lb,eq:scm-induction,eq:nnscm-beta-lb} all involve objective functions or constraints evaluated over a set of parameter values $\bs{P}$. In computational practice, these optimizations are performed over a discretization of parameter space. We will denote $\Pi \subset \widetilde{\mathcal{P}}$ as this discrete parameter grid. For implementation (including the full NNSCM description in \Cref{alg:nnscm}) we replace all these continuous optimizations with discrete ones using the parameter set $\Pi$.

\section{Numerical Examples}

\label{sec:numerics}

We consider several parameterized examples below including a parametric Penzl model and semi-discretizations to various parabolic PDEs including a time-fractional heat equation. All examples are implemented with the open-source Julia software package \texttt{ModelOrderReductionToolkit.jl} \cite{mortjl} using the Gurobi solver for linear programming \cite{gurobi} and ARPACK for sparse eigenvalue computations \cite{arpack}. 

\subsection{Computational Details}

For all problems, we restrict to the single-input, single-output case, $p=q=1$ for simplicity. Thus, for the remainder of this section, we neglect the indexing parameter, $k$, and denote $\widetilde{\mathcal{P}} = i\R \times \mathcal{P}$.

\subsubsection{Full order models}
For the PDE problems, we consider the domain $\Omega = (-1,1)^2$ and impose homogeneous Dirichlet boundary conditions. 
Note that for the PDE models, we use $u$ to denote the continuous state variable and $f$ to denote the input while after discretization these are denoted as $\bs x$ and $\bs u$ respectively.
Discretization is performed by centered finite differences with a tensor product grid of $\sqrt{n}$ uniformly-spaced nodes on the interior of $[-1,1]$ for a state size of $n$ in the resulting system. The input is forcing applied on the exterior of the circle centered at the origin with radius $r=0.5$, i.e., applied to discrete nodes $(x_j,y_k)$ such that $x_j^2 + y_k^2 > 0.5^2$. The output is the discrete approximation to the average value of the state, i.e., the integral of the state over the uniform probability measure on $\Omega$. When reporting weak greedy error decay for problems posed on $\Omega$, instead of reporting vector $\ell^2$ errors which change with $n$, we report discrete $L^2$ errors which are the discrete approximation of the $L^2(\Omega)$ norm over the Lebesgue measure of the state. (This effectively means we report standard $\ell^2$ errors over $n$-vectors normalized by $4/\sqrt{n}$.)

\subsubsection{SCM and NNSCM details}\label{sssec:results-nnscm}
For the simplest, symmetric example, the standard SCM algorithm described in \Cref{sssec:standard-scm} is used. Otherwise, a version of NNSCM, slightly modified from what is described in \Cref{sssec:nnscm}, is applied to each problem to solve for lower-bound approximations of the stability factors, $\sigma^\text{LB}_\text{min}(\bs P) \leq \sigma_\text{min}(\bs P)$. The modification is based on some of the theoretical motivation described in \Cref{ssec:theory}, specifically the discussion of the estimate \eqref{eq:sectorial-estimate}: NNSCM tended to struggle when approximating over the entire frequency-domain, especially since $\omega$ can take values over all $\R$. In particular, the standard NNSCM procedure, with a value of $\obs{\bs{P}}$ fixed, resulted in the procedure attempting to enrich with parameter values corresponding to very distant frequency values, which did not help with constructing a meaningful lower bound. To combat this, we manually decomposed the parameter set $\widetilde{\mathcal{P}}$ by splitting across frequency values $\{\overline{\omega}_j\}_{j=0}^M$, forming parameter subdomains,
\begin{align}\label{eq:omega-discretization}
    \widetilde{\mathcal{P}}_j &= i[\overline{\omega}_{j-1}, \overline{\omega}_{j}) \times \mathcal{P}, & j &\in [M].
\end{align}
With this partition, $M$ separate NNSCMs are trained on each parameter subdomain $\widetilde{\mathcal{P}}_j$. Within each subdomain, the frequency values are discretized by five uniformly spaced values. In all examples below, we take $\overline{\omega}_0 = 0$, $\overline{\omega}_1 = 10^{-2}$, and $\overline{\omega}_M=10^3$. Unless otherwise stated, all other $\overline{\omega}$ are logarithmically spaced between $\overline{\omega}_1$ and $\overline{\omega}_M$. This logarithmic spacing is motivated by the discussion in \Cref{ssec:theory} regarding the geometrically increasing lengths along $i\R$. All standard SCM algorithms are trained with nearest neighbor parameter $M_\alpha=20$, and all NNSCM-type algorithms are trained with $M_\alpha=20$, $\texttt{inside}=\texttt{true}$, and $\varphi=0$. (See \Cref{alg:nnscm} for an explanation of these algorithmic hyperparameters.)

After the SCM or NNSCM is trained, a potentially separate parameter grid is used for greedily building the reduced basis, which we call $\widetilde{\Pi}$.

\subsubsection{Real and complex reduced dimensions}
Even in the specialized case that the original parametric LTI system \eqref{eq:lti} is real-valued (e.g., the initial condition and all system matrices are real-valued), the proposed method results in a generally complex-valued reduced basis and hence a complex-valued reduced system \eqref{eq:lti_red}. Although this is not an issue mathematically, we discuss below one way to reformulate the procedure, amounting to a frequency-domain post-processing, so that the reduced system \eqref{eq:lti_red} is real-valued.

Given a (generally complex-valued) $\bs{\Phi} \in \C^{n \times r_0}$ that is determined through the previously described RBM procedure \eqref{eq:rbm}, one can ensure a real-valued reduced system by performing proper orthogonal decomposition on the real and imaginary components of the reduced basis, $\tbs \Phi = \left[\text{Re}(\bs \Phi) \;\; \text{Im}(\bs \Phi)\right] \in \R^{n \times 2r_0}$. Intuitively, for a general dimension-$r_0$ complex reduced basis, a real basis of dimension-$2r_0$ will be required to encapsulate the same information. However, we find that $\tbs \Phi$ is often numerically near-rank-deficient, so that choosing real rank $r < 2 r_0$ can be effective in practice. 

We will describe results as a function of $(r_0, r)$, which means that we run the RBM procedure \eqref{eq:rbm} to identify a rank-$r_0$ complex-valued reduced model, but numerically use a rank-$r$ real-valued reduced model using the POD truncation strategy described in the previous paragraph, with $r \leq 2 r_0$. Precisely, given a complex reduced-basis matrix $\bs\Phi_c \in \C^{n\times r_0}$, we first form the real SVD of the real and imaginary parts
\[\bs U \bs \Sigma \bs V^T = [\text{Re}(\bs\Phi_c)\;\;\text{Im}(\bs\Phi_c)],\]
and choose a POD energy truncation of $10^{-2}$, selecting,
\begin{align*}
    r &= \min\left\{k\in[2r_0]\;\bigg|\;\sqrt{\frac{\sum_{j=k+1}^n{\Sigma_{j,j}}^2}{\sum_{j=1}^n{\Sigma_{j,j}}^2}} \leq 10^{-2}\right\}, & \bs\Phi = \bs U_{:,1:r} \in \R^{n\times r}.
\end{align*}
The above ensures that $\bs{\Phi}$ is a rank-$r$ real-valued matrix. 
The truncation value $10^{-2}$ is not particularly special, it is chosen to balance projection error with model truncation.

\subsection{Symmetric Example}\label{ssec:symmetric}

The first example we consider is the parametric symmetric parabolic problem found in \cite{chen_certified_2016} given by
\[
    \begin{aligned}
        u_t &= u_{xx} + p_1 u_{yy} + p_2 u + f(x), & x \in \Omega,\\
        p_1 &\in [0.1,4],\;p_2 \in [0,2]. &
    \end{aligned}
\]
The frequency parameter is discretized into 50 logarithmically spaced values between $10^{-2}$ and $10^3$. 
The remaining parameters, $p_1$ and $p_2$, are discretized uniformly with 20 values each between their respective minimum and maximum values for a total of $50\times 20^2 = 2\cdot 10^4$ parameter values. The standard SCM algorithm is used and is trained to a relative gap of $\epsilon=0.8$. This parameter discretization is used both for training the SCM and for greedy reduced basis construction.

The results in \Cref{fig:symmetric} demonstrate consistent exponential error decay in the complex reduced basis dimension. Fixing a complex reduced basis dimension $r_0=10$, a real reduced basis is constructed of order $r=13$. We observe from the Bode plots that the transfer function is approximated to at-worst order $10^{-2}$ relative error. Perhaps unsurprisingly, all $r_0=10$ parameter samples occur at $p_2=2$ corresponding to when the growth term, $p_2 u$, is maximized. 

\begin{figure}[ht!]
    \centering
    \includegraphics[width=0.325\textwidth]{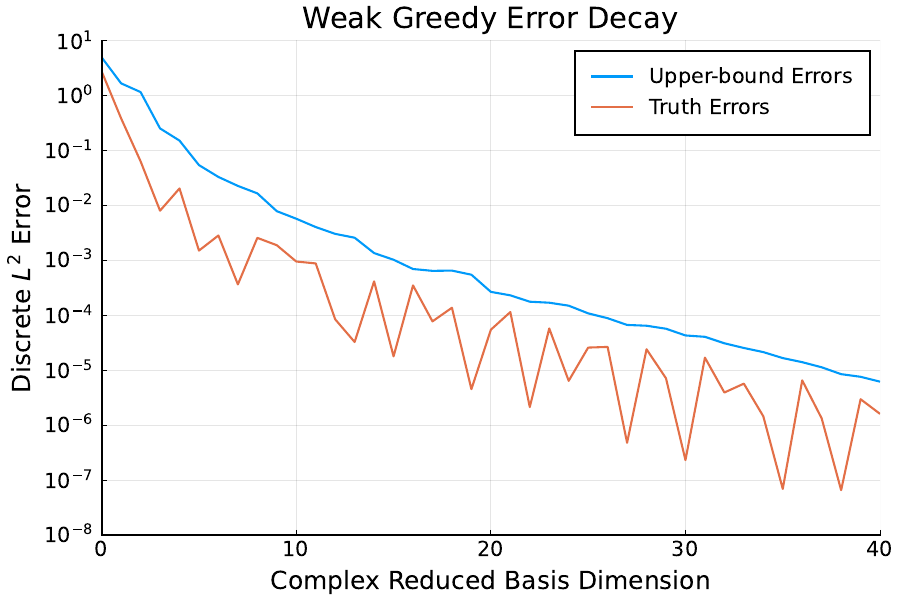}
    \includegraphics[width=0.325\textwidth]{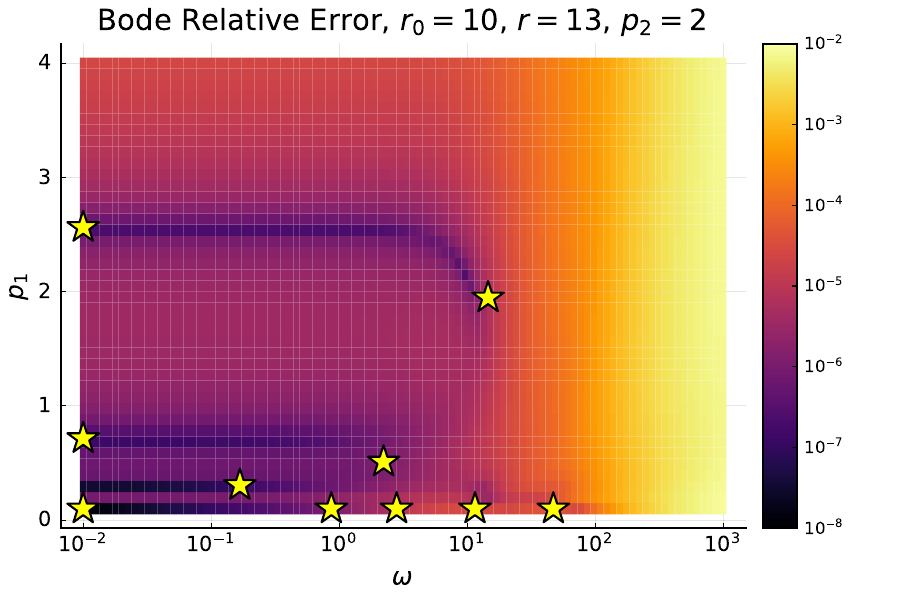}
    \includegraphics[width=0.325\textwidth]{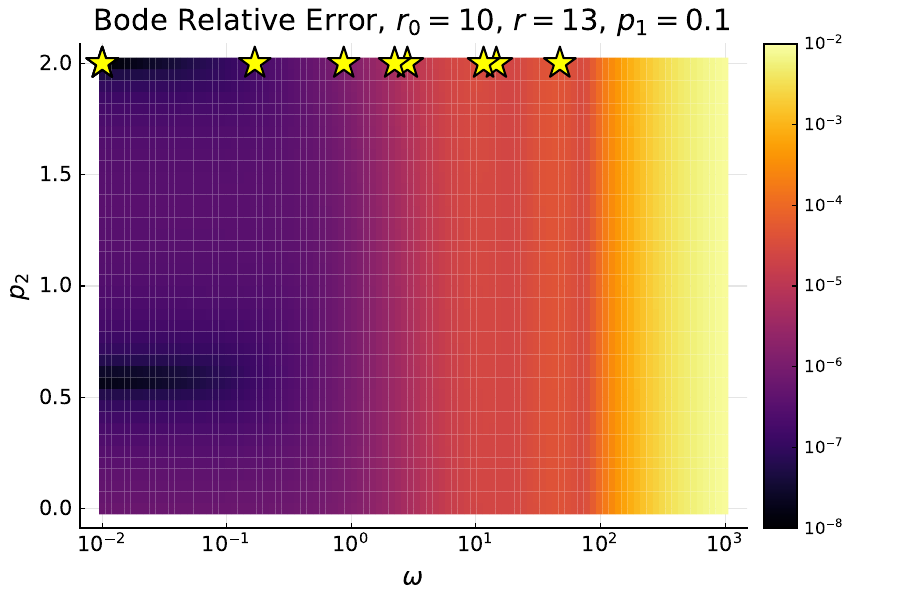}
    \caption{Symmetric example of \cref{ssec:symmetric}, $n=100^2$. Left: Upper bound errors refer to $\max_{\bs{P}} \Delta_{\bs\Phi}(\bs{P})\cdot4/\sqrt{n}$, and truth errors are $\max_{\bs{P}} \|e_{\bs{\Phi}}(\bs{P})\|_2\cdot4/\sqrt{n}$. Middle: Fixed-$p_2$ relative error in the transfer function (Bode plot); yellow stars indicate parameter values corresponding to snapshots chosen by RBM. Right: Fixed-$p_1$ relative error.}
    \label{fig:symmetric}
\end{figure}

\subsection{Parameterized Penzl Example}\label{ssec:ppenzl}
The second example we consider is a three-parameter Penzl model found in \cite{penzl_algorithms_2006,ionita_data-driven_2014} which is a discrete LTI toy-model with parameter-dependent matrices given by
\begin{align*}
    &\begin{aligned} \bs A =
        \text{diag}\bigg(&\begin{pmatrix}
            -1 & 100 + p_1\\
            -100 - p_1 & -1\\
        \end{pmatrix}, \begin{pmatrix}
            -1 & 200 + p_2\\
            -200 - p_2 & -1\\
        \end{pmatrix}, \\
        &\begin{pmatrix}
            -1 & 400 + p_3\\
            -400 - p_3 & -1\\
        \end{pmatrix},\text{diag}(-1,-2,\ldots,-1000)\bigg) \in \R^{1006\times1006},
    \end{aligned} \\
    & \bs B = (10,10,10,10,10,10,1,1,\ldots,1)^T \in \R^{1006\times 1},\\
    & \bs C = \bs B^T \in \R^{1 \times 1006},
\end{align*}
where $p_1,p_2,p_3 \in [-20,20]$. The parameter set is decomposed into $M=19$ frequency ranges, but the values of $\overline{\omega}_j$ are manually chosen as:
\begin{align*}
    \{\overline{\omega}_j\}_{j=0}^M = \{ &0,10^{-2},1,50,80,100,120,150,180,200,220, \\
    &250,300,350,380,400,420,450,500,10^3\}.
\end{align*}
This manual choice is made in order to capture the excited Penzl frequency modes at $\omega=100,200,400$. The remaining parameters are discretized into 9 uniformly spaced points between $-20$ and $20$ inclusive.
The NNSCMs are trained with tolerances $\epsilon=0.6$ and $\epsilon_\beta=0.99$. After the NNSCMs are trained, the greedy reduced basis is constructed via discretization of the frequency parameter with fifty logarithmically spaced points between $10^{-2}$ and $10^3$, and the other parameters as previously described.

The results are displayed in \Cref{fig:penzl}. We observe that this model observes a fast exponential rate of decay in the weak greedy error. Additionally, the lower-bound and truth stability factors are plotted at two parameter values, the latter of which is off the parameter grid. We observe relatively tight lower-bound predictions in each. A complex reduced basis dimension $r_0=15$ is chosen resulting in a real reduced basis dimension of $r=20$. For this dimension ROM, the bode plots display at-worst order $10^{-2}$ relative error in the transfer function. The relative error is lowest near the excited frequencies, $100+p_1$, $200+p_2$, and $400+p_3$.

\begin{figure}[ht!]
    \centering
    \includegraphics[width=0.325\textwidth]{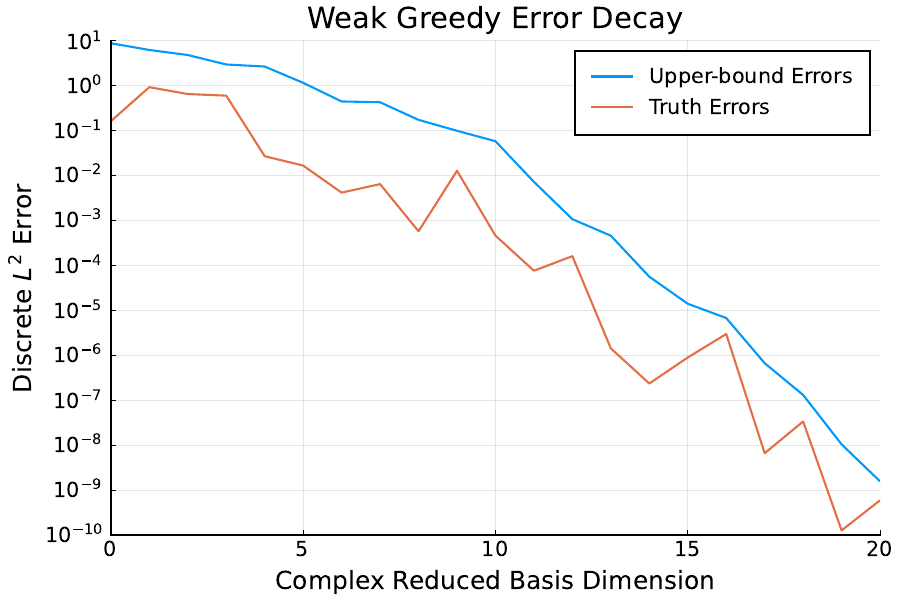}
    \includegraphics[width=0.325\textwidth]{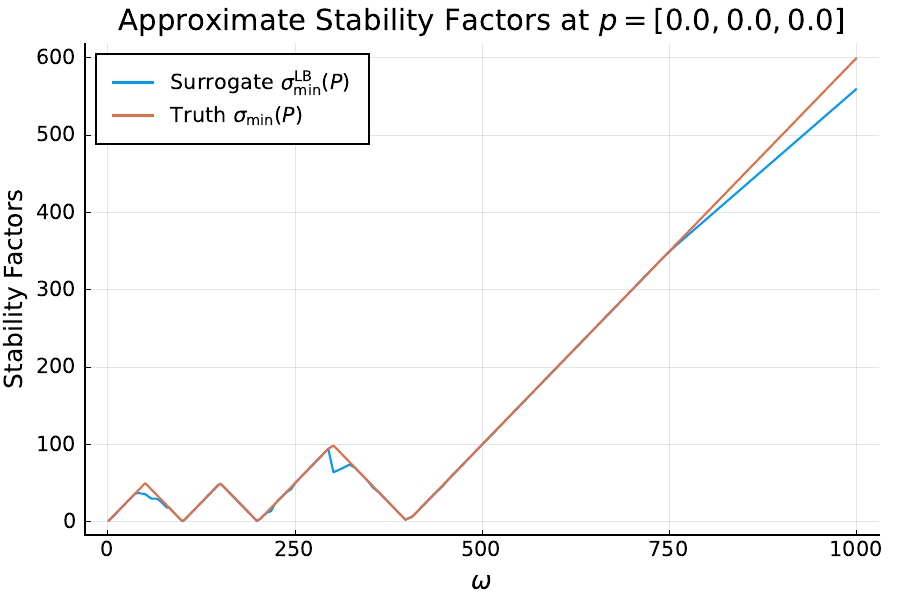}
    \includegraphics[width=0.325\textwidth]{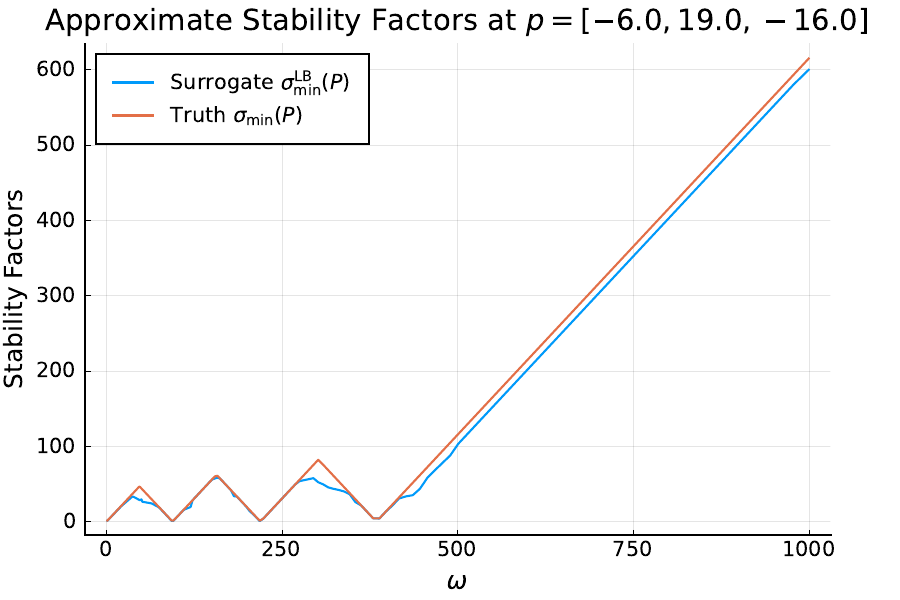}
    \includegraphics[width=0.325\textwidth]{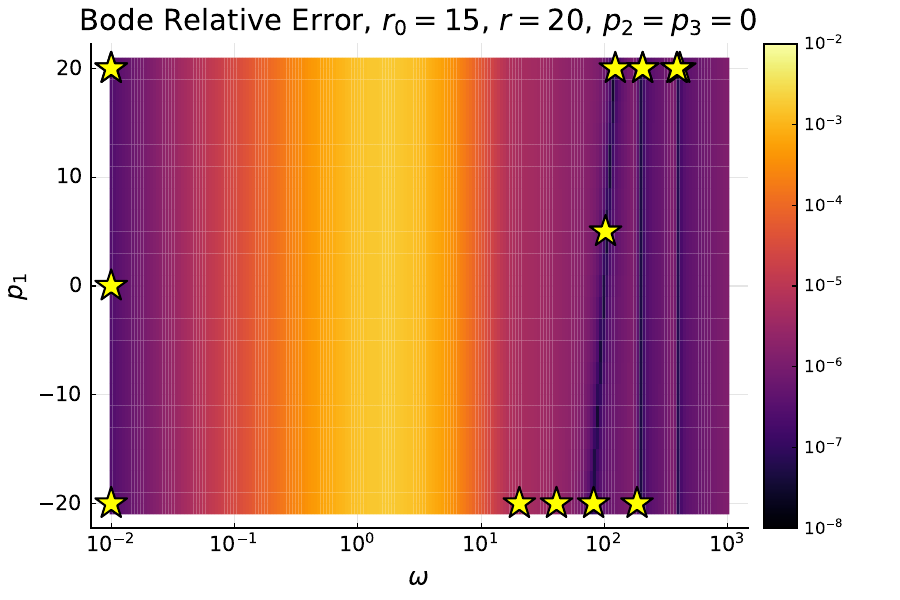}
    \includegraphics[width=0.325\textwidth]{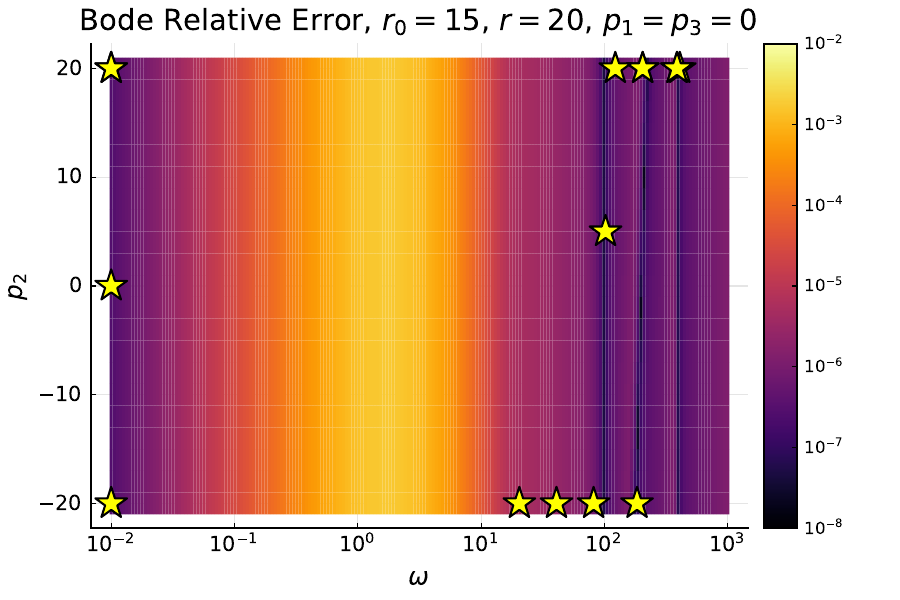}
    \includegraphics[width=0.325\textwidth]{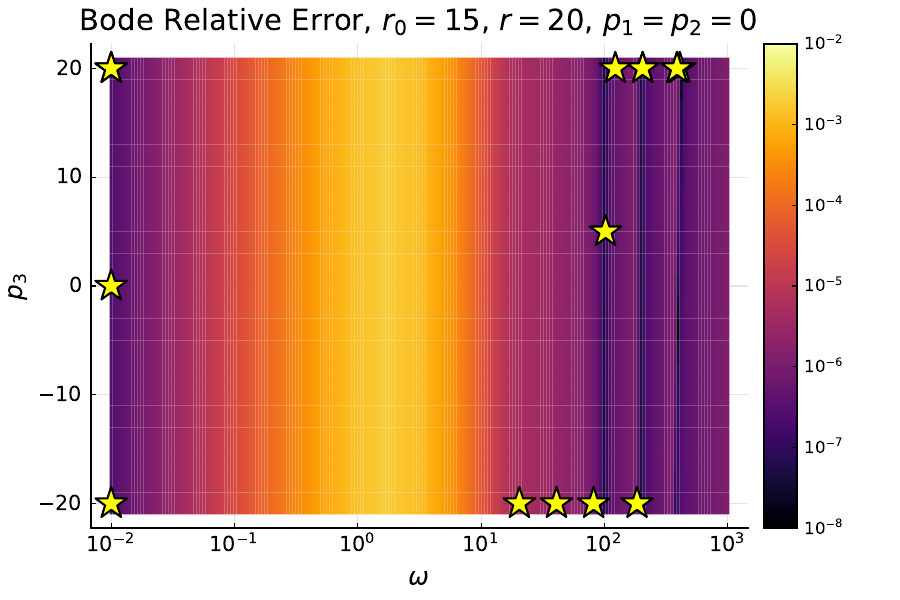}
    \caption{Parametric Penzl example of \cref{ssec:ppenzl}, $n=1006$. Top row: 
    (left) Upper bound errors refer to $\max_{\bs{P}} \Delta_{\bs\Phi}(\bs{P})$, and truth errors are $\max_{\bs{P}} \|e_{\bs{\Phi}}(\bs{P})\|_2$. (middle, right): NNSCM-computed stability factors $\sigma_{\min}^{\text{LB}}$ versus actual value $\sigma_{\min}$ at different values of $(p_1, p_2, p_3)$. Bottom row: Relative error in the transfer function (Bode plot) for various 2D slices of $(\omega, p_1, p_2, p_3)$; yellow stars indicate parameter values corresponding to snapshots chosen by RBM.}
    \label{fig:penzl}
\end{figure}

\subsection{Vanishing Diffusion Example}\label{ssec:vandiff}

The third example we consider is the vanishing diffusion problem found in \cite{chen_certified_2016} given by
\[
    \begin{aligned}
        u_t &= (1 + p_1 x) u_{xx} + (1 + p_2 y) u_{yy} + f(x), & x \in \Omega,\\
        p_1,p_2 &\in [-0.99,0.99]. &
    \end{aligned}
\]
The parameter set is decomposed into $M=6$ frequency ranges determined by $\{\overline{\omega}_j\}_{j=0}^M = \{0,10^{-2},10^{-1},10^{0},10^{1},10^{2},10^{3}\}$. The remaining parameters are discretized with $10$ uniformly spaced values between $-0.99$ and $0.99$ inclusive. The NNSCMs are trained with tolerances $\epsilon=0.8$ and $\epsilon_\beta=0.9999$. The frequency parameter discretization for construction of the greedy reduced-basis is fifty logarithmically spaced points between $10^{-2}$ and $10^{3}$.

The results are displayed in \Cref{fig:vanishingdiffusion}. We observe that this model exhibits slower weak greedy error decay than the previous two problems. The next two panels illustrate the approximate stability surfaces at $\omega=10^{-3}$ and at $\omega=100$, with the surfaces colored by where NNSCM performed domain decomposition. We note that at small frequency values, the true stability surface looks like a paraboloid which is maximized at $p_1=p_2=0$ and decays towards zero as $|p_1|\rightarrow1$ or $|p_2|\rightarrow1$; the NNSCM represented this surface with five domain decompositions forming an almost piecewise linear approximant. At larger frequency values, the stability surface is flatter. We observe that at $\omega=100$, NNSCM took more iterations to converge to a relative gap of $\epsilon=0.8$. Finally, a complex reduced basis of dimension $r_0=10$ is resulting in a real reduced basis of dimension $r=10$. These two being the same is due to the fact that all greedy parameters chosen for full-order solves are at a low frequency value of $10^{-2}$, resulting in mostly-real frequency-domain solutions. We observe that the dimension $r=10$ ROM has a maximum output transfer function relative error of $10\%$ with smaller relative error nearer the sampled parameter values.

\begin{figure}[ht!]
    \centering
    \includegraphics[width=0.325\textwidth]{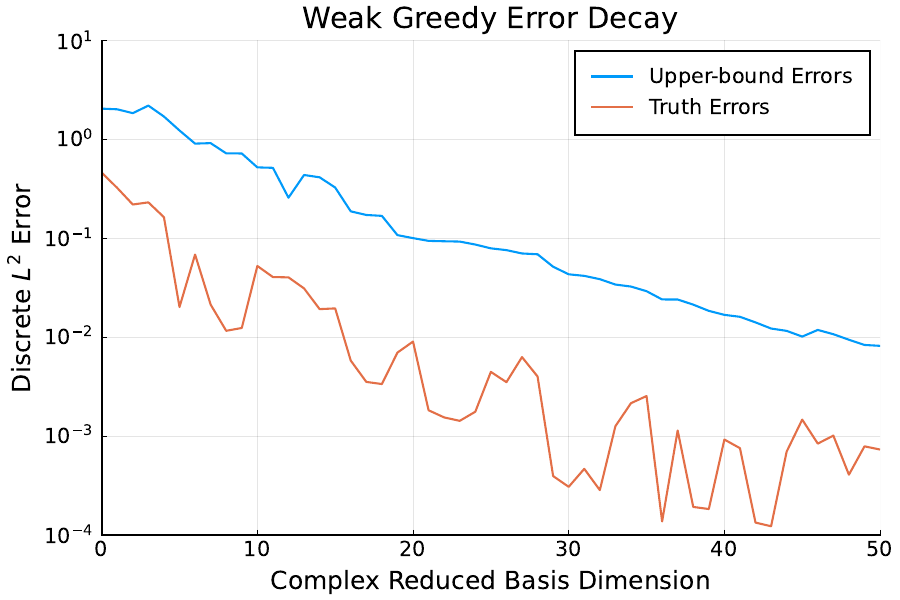}
    \includegraphics[width=0.325\textwidth]{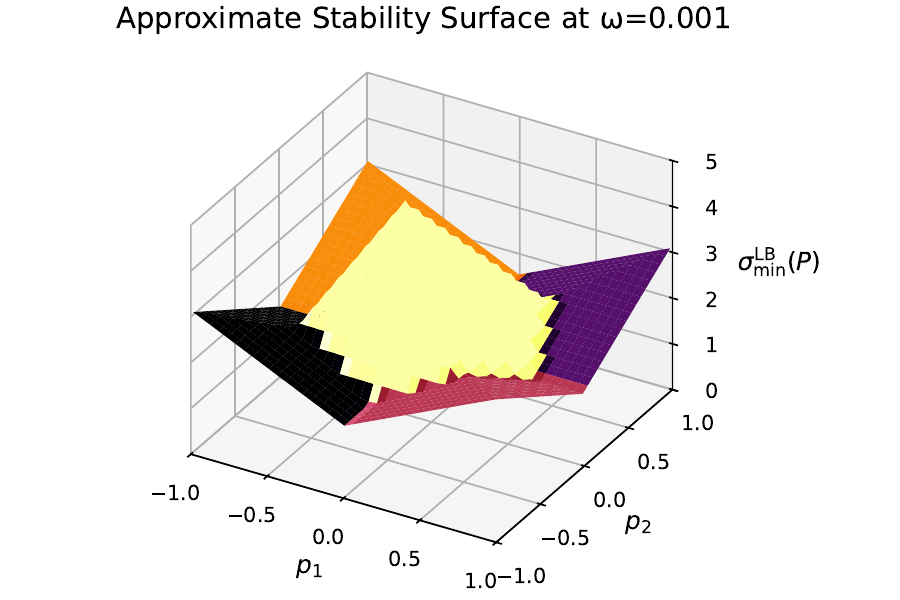}
    \includegraphics[width=0.325\textwidth]{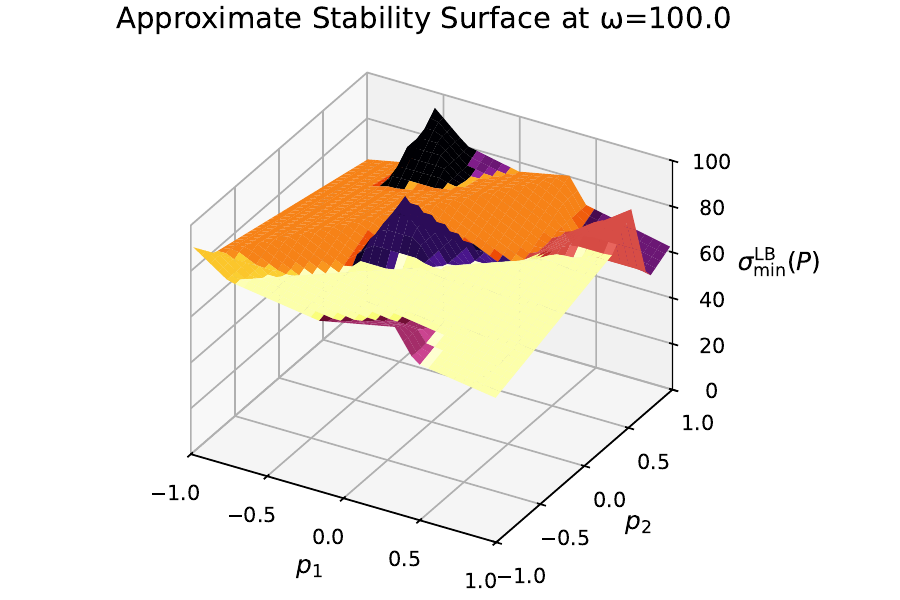}
    \includegraphics[width=0.325\textwidth]{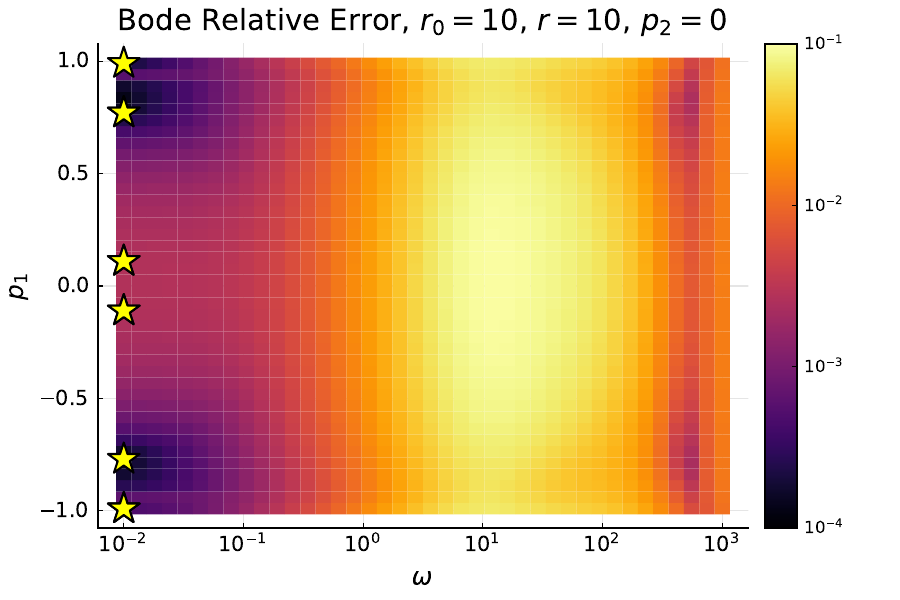}
    \includegraphics[width=0.325\textwidth]{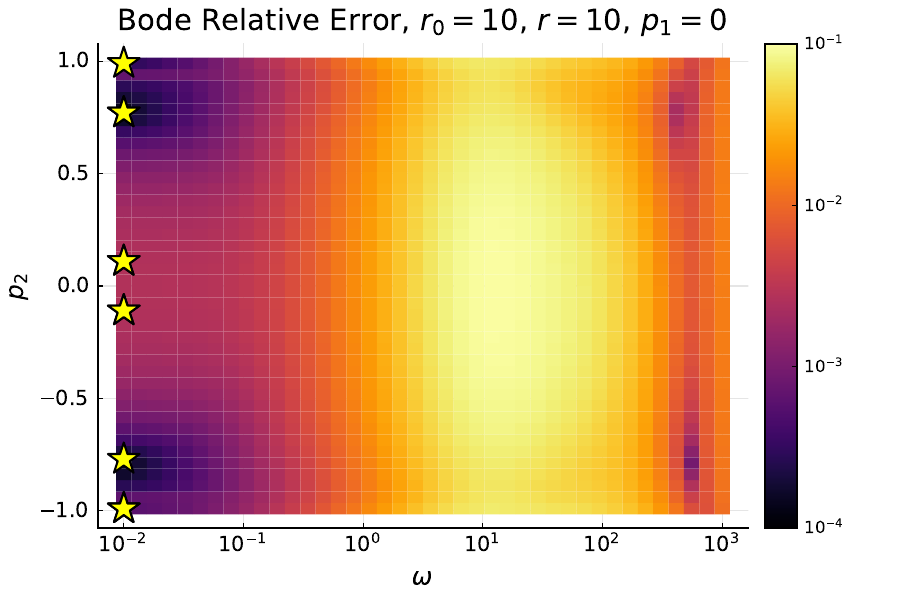}
    \includegraphics[width=0.325\textwidth]{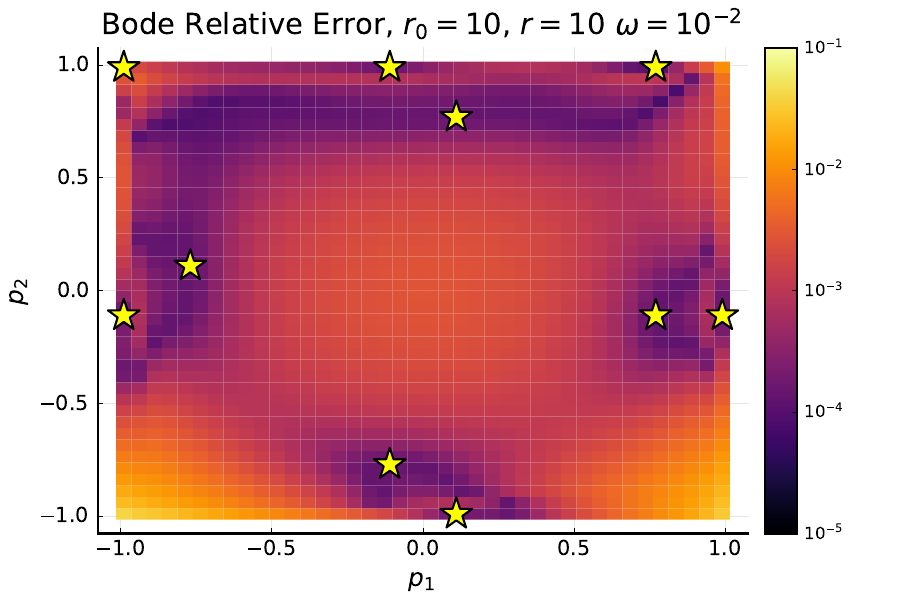}
    \caption{Vanishing diffusion example of \cref{ssec:vandiff}, $n=100^2$. Top row: 
    (left) Upper bound errors refer to $\max_{\bs{P}} \Delta_{\bs\Phi}(\bs{P})\cdot4/\sqrt{n}$, and truth errors are $\max_{\bs{P}} \|e_{\bs{\Phi}}(\bs{P})\|_2\cdot4/\sqrt{n}$. (middle, right): NNSCM-computed stability factors $\sigma_{\min}^{\text{LB}}$ at fixed frequency values; different colors highlight different local linearization choices of $\obs{P}$. Bottom row: Relative error in the transfer function (Bode plot) for various 2D slices of $(\omega, p_1, p_2)$; yellow stars indicate parameter values corresponding to snapshots chosen by RBM.}
    \label{fig:vanishingdiffusion}
\end{figure}

\subsection{Time-Fractional Heat Equation}\label{ssec:fracpde}

The final example we consider is a time-fractional heat equation given by
\[
    \begin{aligned}
        u^\alpha_t &= \Delta u + f(x), & x \in \Omega,\\
        \alpha &\in [0.05,1], &
    \end{aligned}
\]
where the time-fractional derivative is the Caputo-type fractional derivative
\begin{align*}
    u_t^\alpha(t) &= \frac{1}{1 - \alpha} \int_0^t \frac{u_t(s)}{(t-s)^\alpha} ds, & \alpha&\in(0,1).
\end{align*}
Alike how the Laplace transform of $u_t$ with zero initial conditions, is given by $s U$, the Laplace transform of the Caputo-type fractional derivative with zero initial conditions is given by $s^\alpha U$ which matches the standard time partial-derivative when $\alpha=1$ \cite{fahad_time_fractional}. For $s = i\omega \in \C$, we have that
\[s^\alpha = (i\omega)^\alpha = (|\omega| e^{i\cdot\text{sign}(\omega)\pi/2})^\alpha = |\omega|^\alpha(\cos(\alpha\pi/2) + i \cdot\text{sign}(\omega) \sin(\alpha\pi/2)),\]
allowing us to write $s^\alpha\bs{I} - \bs{A}$ in affine parameter dependent form where the coefficients are purely real.

The parameter set is again decomposed into $M=6$ frequency ranges, and $\alpha$ is discretized with $20$ uniformly spaced values between $0.05$ and $1.0$ inclusive. 
The NNSCM is trained with tolerances $\epsilon=0.4$ and $\epsilon_\beta=0.9999$. The frequency parameter discretization for construction of the greedy reduced-basis is fifty logarithmically spaced points between $10^{-2}$ and $10^{3}$.

The results are displayed in \Cref{fig:fracdiffusion}. This model obtains very fast exponential weak greedy error decay. The approximate stability surface appears to be linearly increasing for each fixed $\alpha$ with slope increasing in $\alpha$. For this model, we illustrate transfer function relative errors for increasing complex reduced basis dimensions of $r_0=3,4,5,10$. These figures demonstrate a pattern in the parameter values selected approaching $\alpha=1$ and $\omega=10^3$. 

\begin{figure}[ht!]
    \centering
    \includegraphics[width=0.325\textwidth]{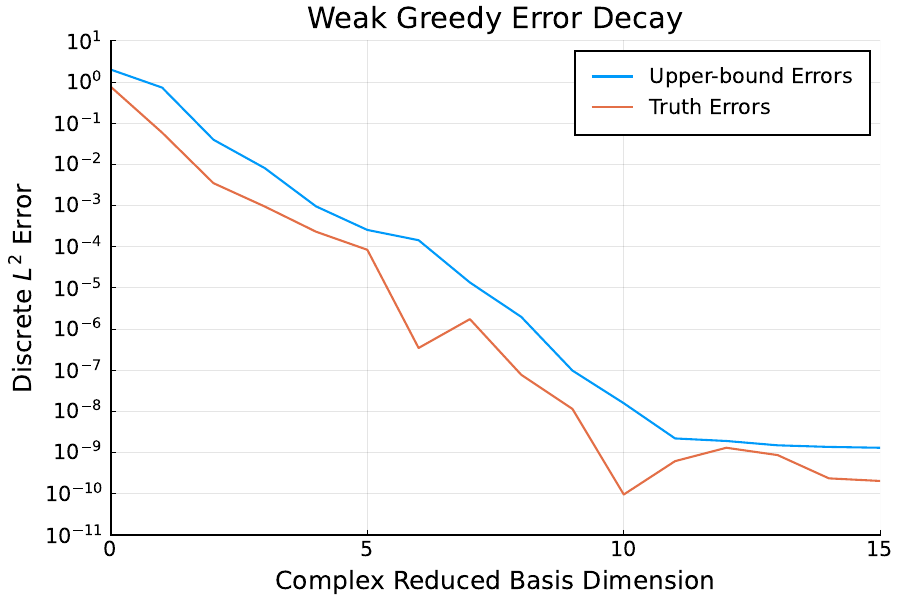}
    \includegraphics[width=0.325\textwidth]{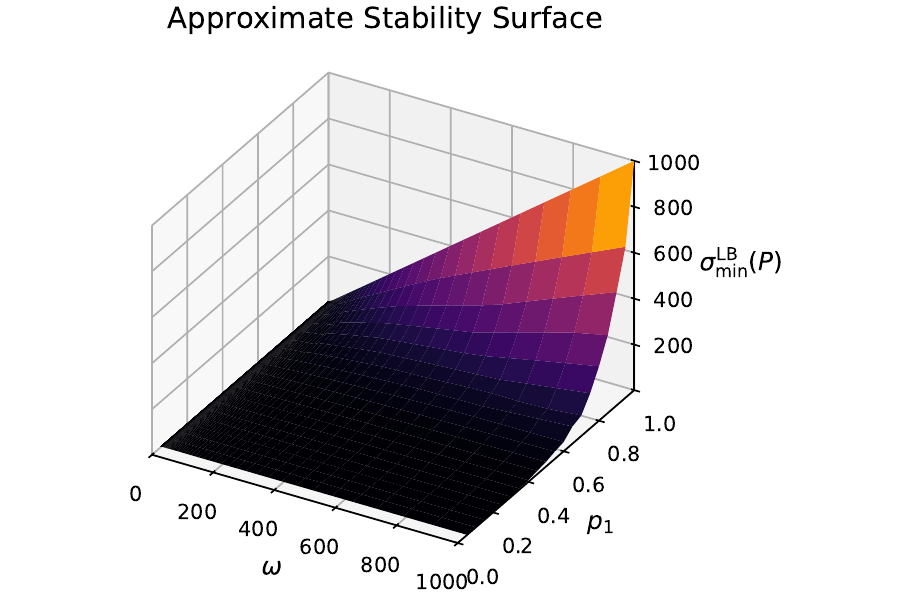}
    \includegraphics[width=0.325\textwidth]{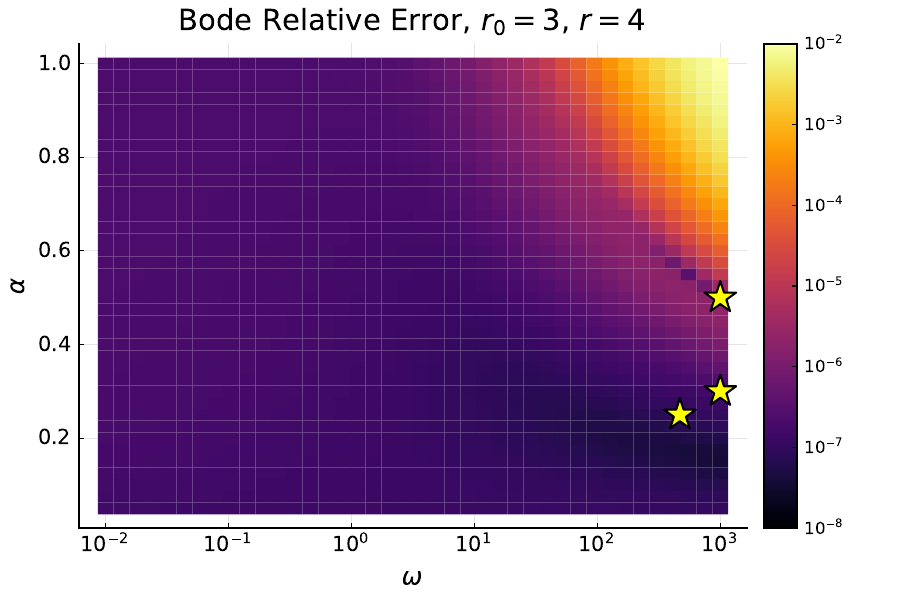}
    \includegraphics[width=0.325\textwidth]{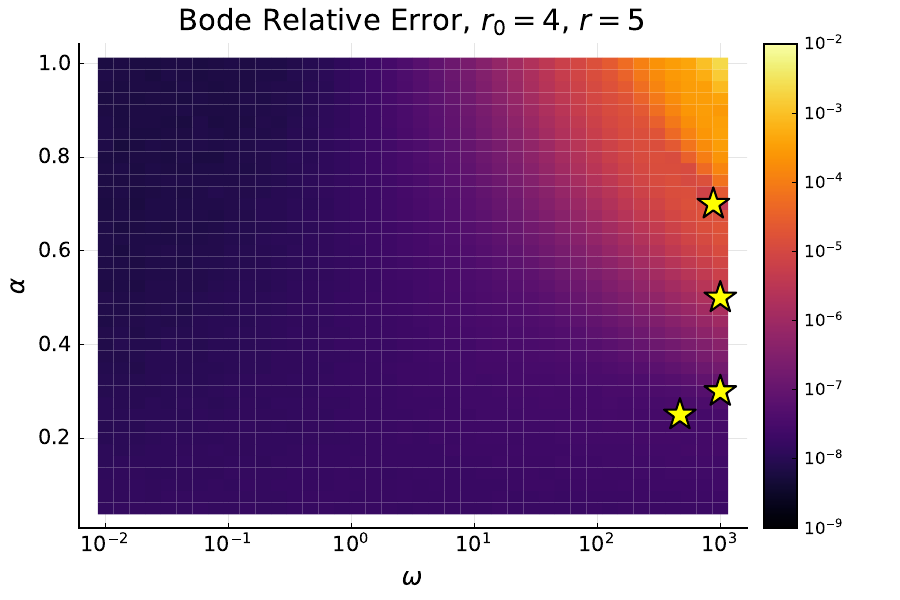}
    \includegraphics[width=0.325\textwidth]{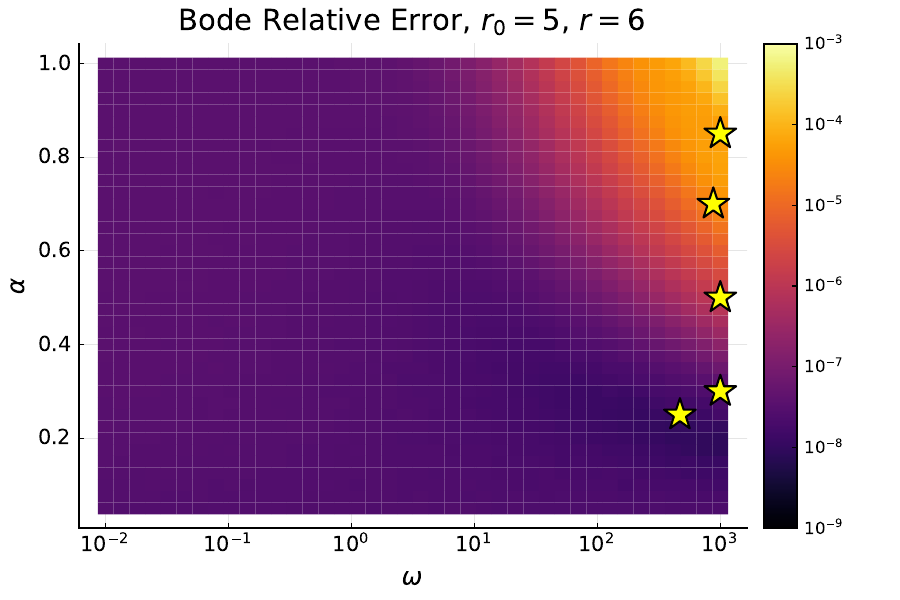}
    \includegraphics[width=0.325\textwidth]{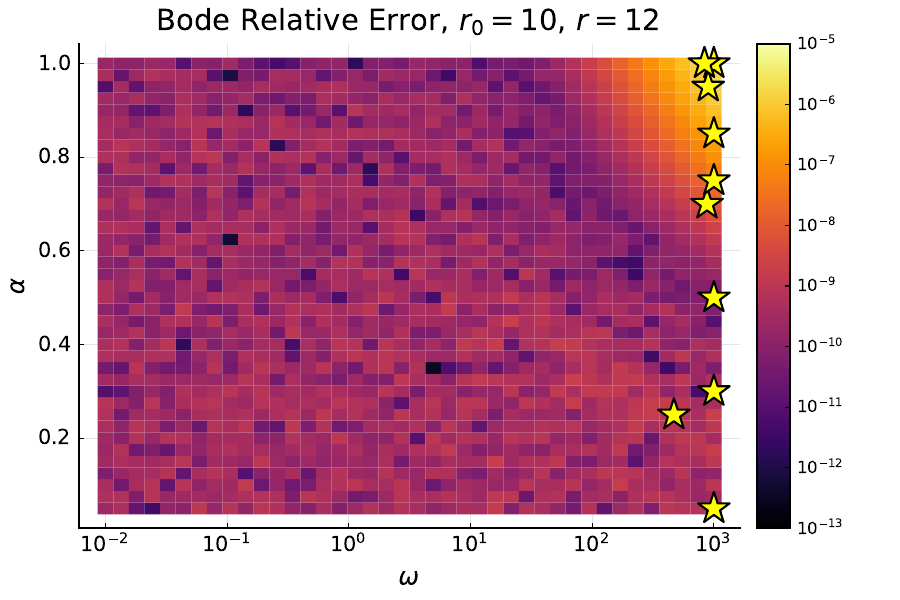}
    \caption{Time-fractional heat equation of \cref{ssec:fracpde}, $n=100^2$. Top row: 
    (left) Upper bound errors refer to $\max_{\bs{P}} \Delta_{\bs\Phi}(\bs{P})\cdot4/\sqrt{n}$, and truth errors are $\max_{\bs{P}} \|e_{\bs{\Phi}}(\bs{P})\|_2\cdot4/\sqrt{n}$. (middle) NNSCM-computed stability factors $\sigma_{\min}^{\text{LB}}$. 
    Top row (right) and bottom row: Relative error in the transfer function (Bode plot) for increasing values of $r_0$; yellow stars indicate parameter values corresponding to snapshots chosen by RBM.}
    \label{fig:fracdiffusion}
\end{figure}

\section{Conclusion}
\label{sec:conclusion}
We have explored a scheme that greedily constructs a (multivariate) rational function approximation to the transfer function of parametric LTI systems. Under suitable assumptions,  this approximation is known to converge at rates comparable to decay of a Kolmogorov $n$-width. We have proposed and numerically investigated a modified RBM algorithm, suitable for complex-valued and unbounded parameter sets, that implements the greedy scheme. We show that this procedure is effective on a wide variety of parametric LTI problems.

\section{Acknowledgments}
AN and FB were partially supported by AFOSR FA9550-23-1-0749. AN was partially supported by NSF DMS-2529303. YC was partially supported by NSF DMS-2208277 and AFOSR FA9550-25-1-0181.

\section{Data Availability Statement}
The authors confirm that no databases were used for this paper and that the figures can be reproduced using the open-source software package \texttt{ModelOrderReductionToolkit.jl} \cite{mortjl} and the Gurobi solver \cite{gurobi} along with the hyperparameters detailed in the manuscript. 

\printbibliography

\newpage
\appendix
\section{NNSCM Algorithm}\label{sec:nnscm-alg}

\begin{algorithm}
    \caption{NNSCM}\label{alg:nnscm}
    \begin{flushleft}
        \textbf{Input: } Affine parameter dependent matrix $\bs M(\bs P)$, parameter grid $\widetilde{\Pi}$, stability parameter $M_\alpha \in \Z_{\geq0}$, sample flag \texttt{inside}, tolerances $\epsilon,\epsilon_{\beta},\varphi$ \\
        \textbf{Output: } NNSCM approximant of $\sigma_\text{min}(\bs p)$
    \end{flushleft}
    \begin{algorithmic}[1]
        \State Compute $\sigma_\text{max}(\bs M_j)$ for $j=1,\ldots,Q_A$ \Comment{$Q_A$ eigenvalue problems}
        \State $\mathcal{C}_0 \gets \emptyset$
        \For{$k=1,2,\ldots$}
            \If{$\max_{\bs P \in \widetilde{\Pi}} \epsilon(\bs p) < \epsilon$}
                \State Break
            \EndIf{}
            \State Select $\obs P_k = \argmax_{\bs P \in \widetilde{\Pi}} \epsilon(\bs P)$ \Comment{Other greedy approaches exist \cite{chen_certified_2016}}
            \State $\mathcal{C}_k \gets \mathcal{C}_{k-1} \cup \{\obs P_k\}$
            \State Compute $\sigma_\text{min}(\obs P_{k})$, $\bs z(\bs v(\obs P_k, \obs P_k),\obs P_k)$ \Comment{Eigenvalue problem}
            \State $\mathcal{C}_1(\obs P_k) \gets \{\obs P_k\}$ \Comment{Form $\obs P_k$-subsample set}
            \State $\mathcal{D}_0 \gets \emptyset$
            \For{$j=1,2,\ldots$}
                \State Compute $\epsilon_\beta(\bs P, \obs P_k, \mathcal{C}_j(\obs P_k))$ for each $\bs P \in \widetilde{\Pi}$ \Comment{$|\widetilde{\Pi}|$ LPs}
                \State $\mathcal{D}_j \gets \{\bs P \in \widetilde{\Pi}\;\mid\beta^\text{LB}(\bs P, \obs P_k, \mathcal{C}_j(\obs P_k))>\varphi\}$ \Comment{Current domain}
                \If{$\mathcal{D}_j \setminus \mathcal{D}_{j-1} = \emptyset$ and $\max_{\bs P \in \mathcal{D}_j} \epsilon_\beta(\bs P, \obs P_k, \mathcal{C}_j(\obs P_k)) < \epsilon_\beta$}
                    \State Break \Comment{Domain did not increase and sufficiently constrained}
                \EndIf{}
                \State $\bs P_{j,1} \gets \argmax_{\bs P\in\widetilde{\Pi}} \epsilon_\beta(\bs P, \obs P_k, \mathcal{C}_j(\obs P_k))$
                \State Compute $\beta(\bs P_{j,1},\obs P_k)$, $\bs z(\bs v(\bs P_{j,1}, \obs P_k),\obs P_k)$ \Comment{Eigenvalue problem}
                \State $\mathcal{C}_{j+1}(\obs P_k) \gets \mathcal{C}_j(\obs P_k) \cup \{\bs P_{j,1}\}$
                \If{\texttt{inside}} \Comment{Also enrich inside $\mathcal{D}_j$}
                    \State $\bs P_{j,2} \gets \argmax_{\bs P\in\mathcal{D}_j} \epsilon_\beta(\bs P, \obs P_k, \mathcal{C}_j(\obs P_k))$
                    \If{$\epsilon_\beta(\bs P_{j,2}, \obs P_k, \mathcal{C}_j(\obs P_k)) > \epsilon_\beta$}
                        \State Compute $\beta(\bs P_{j,2},\obs P_k)$, $\bs z(\bs v(\bs P_{j,2}, \obs P_k),\obs P_k)$ \Comment{Eigenvalue problem}
                        \State $\mathcal{C}_{j+1}(\obs P_k) \gets \mathcal{C}_{j+1}(\obs P_k) \cup \{\bs P_{j,2}\}$
                    \EndIf{}
                \EndIf{}
            \EndFor
        \EndFor{}
    \end{algorithmic}
\end{algorithm}

\end{document}